\documentclass[12pt]{elsarticle}

\usepackage[latin1,utf8]{inputenc}
\usepackage[english]{babel}
\usepackage{graphicx}
\usepackage{amsmath,amsfonts,amssymb,amsthm}
\usepackage{subfigure}
\usepackage{xcolor}
\usepackage{bm}
\usepackage{booktabs}
\usepackage{changes}
\usepackage[top=2cm]{geometry}
\usepackage{rotating}
\usepackage{mathtools}
\definecolor{blun}{cmyk}{0.8, 0.5, 0, 0.7}
\usepackage[bookmarks,colorlinks,pdfauthor={Cinzia Soresina},linkcolor=blun,citecolor=blun]{hyperref}
\renewcommand{\H}{\mathcal{H}}
\newcommand{\R}{\mathcal{R}}
\newcommand{\A}{\mathcal{A}}
\newcommand{\x}{\bm{x}}
\newcommand{\y}{\bm{y}}

\newcommand{\llambda}{\bm{\lambda}}
\newcommand{\ls}{\lambda_S}
\newcommand{\dls}{\dot{\lambda}_S}
\newcommand{\li}{\lambda_I}
\newcommand{\dli}{\dot{\lambda}_I}
\newcommand{\uv}{u_v}
\newcommand{\uc}{u_c}
\newcommand{\uq}{u_i}
\newcommand{\ur}{u_r}
\newcommand{\uu}{u_{l}}
\newcommand{\au}{\alpha_1}
\newcommand{\ad}{\alpha_2}
\newcommand{\vmax}{\uv^{max}}
\newcommand{\cmax}{\uc^{max}}
\newcommand{\qmax}{\uq^{max}}
\newcommand{\rmax}{\ur^{max}}
\newcommand{\umax}{u^{max}}
\newcommand{\Tumax}{T_{\tau=0}}
\newcommand{\interv}{starting intervention }
\theoremstyle{plain}
\newtheorem{teo}{Theorem}
\newtheorem{oss}{Remark}

\newtheorem{corollario}{Corollary}

\theoremstyle{definition}
\newtheorem{defin}{Definition}

\newtheorem{teoapp}{Theorem}

\journal{Mathematical Biosciences}

\begin{document}

\begin{frontmatter}

\title{Time-optimal control strategies in SIR epidemic models}

\author[IZSLER]{Luca Bolzoni\corref{mycorrespondingauthor}}\cortext[mycorrespondingauthor]{Corresponding author}\ead{luca.bolzoni@izsler.it}
\author[UniPr]{Elena Bonacini}
\author[UniMi]{Cinzia Soresina}
\author[UniPr]{Maria Groppi}

\address[IZSLER]{Risk Analysis Unit, Istituto Zooprofilattico Sperimentale della Lombardia e dell'Emilia Romagna, Via dei Mercati 13, 43126 Parma, Italy}
\address[UniPr]{Department of Mathematics and Computer Science, University of Parma, Parco Area delle Scienze 53/A, 43124 Parma, Italy }
\address[UniMi]{Department of Mathematics ``F. Enriques'', University of Milano,\\ Via C. Saldini 50, 20133 Milano, Italy}

\begin{abstract}

We investigate the time-optimal control problem in SIR (Susceptible-Infected-Recovered) epidemic models, focusing on different control policies: vaccination, isolation, culling, and reduction of transmission.
Applying the Pontryagin's Minimum Principle (PMP) to the unconstrained control problems (i.e. without costs of control or resource limitations), we prove that, for all the policies investigated, only bang-bang controls with at most one switch are admitted.
When a switch occurs, the optimal strategy is to delay the control action some amount of time and then apply the control at the maximum rate for the remainder of the outbreak.\\
This result is in contrast with previous findings on the unconstrained problems of minimizing the total infectious burden over an outbreak, where the optimal strategy is to use the maximal control for the entire epidemic.
Then, the critical consequence of our results is that, in a wide range of epidemiological circumstances, it may be impossible to minimize the total infectious burden while minimizing the epidemic duration, and vice versa.\\
Moreover, numerical simulations highlighted additional unexpected results, showing that the optimal control can be delayed also when the control reproduction number is lower than one and that the switching time from no control to maximum control can even occur after the peak of infection has been reached.
Our results are especially important for livestock diseases where the minimization of outbreaks duration is a priority due to sanitary restrictions imposed to farms during ongoing epidemics, such as animal movements and export bans.

\end{abstract}


\end{frontmatter}

\section{Introduction}
The emergence and re-emergence of infectious diseases represent a major threat to public health and may cause heavy economic and social losses. Recent epidemics of Ebola in West Africa and MERS-CoV in South Korea highlighted once again the requirement for strong public health interventions for fast disease eradication \cite{dong_hyun15, kucharski15}.

In a similar way, outbreaks of infectious diseases in domestic animals may cause significant consequences for both the sustainability of the livestock industry and the costs associated to disease surveillance, control, and eradication. Moreover, the economic burdens imposed by livestock diseases exceed the agricultural compartment, by affecting also commerce, tourism, and even human health in the infected areas. Consequently, minimizing the time period needed for outbreaks eradication in the affected areas represents a public health priority.

There exist several examples of livestock epidemics causing huge sanitary and economic impacts, such as the 1996 epidemic of classical swine fever in The Netherlands \cite{meuwissen99}, the 2001 epidemic of foot-and-mouth in the UK \cite{DEFRA}, and the 2015 epidemic of high pathogenic avian influenza in Midwestern USA \cite{hvistendahl15}.
From the epidemiological point of view, the main indicators generally used to describe the severity of these infection events in livestock are: (\textit{i}) the total number of infected animals and farms during an epidemic, and (\textit{ii}) the duration of the epidemic. The rationale behind these indicators is based on the evidence that epidemic surveillance and control costs are directly related to spatial and temporal extension of the epidemic events \cite{horst99}. Furthermore, the effect of the epidemic duration on the socio-economic burdens associated to livestock diseases is larger than in human diseases. This is due to the sanitary restrictions imposed to farms in infected areas during ongoing outbreaks, such as animal movement and export bans. Moreover, the block or the restriction of farm activities can go over the time of infection, carrying on until the disease-free status is formally regained \citep{hutber2006predictions}. Examples of costly restrictions for the livestock industry include: the export ban of UK cattle because of the 1996 bovine spongiform encephalopathy epidemic \citep{gordon2008bovine} and the export ban of poultry and poultry related products in Hong Kong, Laos, Thailand, and The Netherlands due to outbreaks of highly pathogenic avian influenza \citep{koopmans04, nerlich07, otte08}.

By using a stochastic modeling framework for classical swine fever in The Netherlands pig farms, Mangen et al. \cite{mangen04} showed that the increase of the epidemic duration affects the sanitary costs associated to disease outbreaks more than a proportional growth in the number of infected farms. This prediction followed from the observation that longer epidemics are more widespread, involving a larger number of animals slaughtered. The estimate of the epidemic duration appears almost invariably in the simulation outputs of data-driven mathematical models developed to evaluate the effectiveness and the efficiency of surveillance and control policies for several infections in livestock, such as foot-and-mouth disease \citep{roche15}, classical swine fever \citep{durr13}, bovine tuberculosis \citep{rossi15}, and avian influenza \citep{longworth14}. However, few attempts have been made to address the problem of minimizing the epidemic duration from a theoretical point of view by using optimal control theory. To our knowledge, the only example of analytic characterization of the control function in a time-optimal framework is due to Jiang \cite{jiang07}, who focused on the analysis of isolation strategies in a subsystem of the model proposed in Zhang et al. \cite{zhang05} to describe SARS spread. On the other hand, the optimal control theory has been widely applied to solve the problem of minimizing the total number of infected individuals (or the total infectious burden) in basic SIR (Susceptible-Infected-Recovered) epidemic models by means of different control policies, such as: the implementation of emergency prophylactic vaccination plans, the isolation of infected individuals, the reduction of disease transmission through the limitation of contacts between individuals, and non-selective culling \citep{abakuks73, abakuks74, morton74, wickwire75, behncke00, hansen11JOMB, bolzoni14}.

Prophylactic vaccination consists in the vaccination of susceptible individuals; its goal is to prevent the development of diseases.
Isolation consists in the quarantine of infected individuals. As regards livestock diseases, in SIR models isolation is mathematically equivalent to removal of infected individuals through test-and-cull procedures.
Non-selective culling consists in the slaughtering of both infected and healthy individuals and it is usually implemented in wildlife and livestock when no other options are available (e.g. no diagnostic tests available, lack of time or resources). The rationale for culling healthy individuals resides in the positive relationship between the rate at which individuals become infected and the abundance of susceptible individuals.
Among humans, the reduction of transmission can be obtained through information campaigns or emergency movement bans (e.g. school closures, flight limitations), while in livestock it can be obtained by imposing limitations on animal, vehicle, and personnel movements among farms.

Those cited studies solved the optimal control problem for the minimization of the infectious burden in unconstrained conditions (i.e. without costs of control or resource limitations). They showed that the optimal strategy always relies in the adoption of the maximum control for the entire epidemic. In this context, maximum control is intended as the implementation of the control policy at its maximum available rate.

Here, by using simple SIR models in an optimal control framework \citep{pontryagin}, we thoroughly investigate the problem of minimizing the epidemic duration by using prophylactic vaccination, isolation, non-selective culling, or reduction of transmission controls.
In this study, we will show that the optimal control strategies to minimize the epidemic duration in SIR models can substantially differ from those minimizing the infectious burden.
Specifically, we will prove that: (\textit{i}) using the maximum control for the entire epidemic duration may not be the optimal strategy (even in unconstrained conditions); and (\textit{ii}) when the maximum control for the entire epidemic is not an optimal strategy, a delayed control is optimal. Consequently, our results lead to the conclusion that minimizing the epidemic duration does not always imply minimizing the total infectious burden, and vice versa.

\section{Optimal control problem: general setting} \label{stan}
We describe the evolution of the infection in a host population with a standard deterministic SIR model \citep{anderson79}, that can be described by the following system of two ordinary differential equations (ODEs):
\begin{equation}
\begin{cases}
\dot{S} = f_1(S,I) = -\beta SI\\
\dot{I} = f_2(S,I) =  \beta SI - \mu I,
\end{cases}\label{SI}
\end{equation}
where $S(t)$ and $I(t)$ represent the number of susceptible and infected hosts, respectively, $\beta$ represents the transmission rate of the infection and $\mu$ represents the loss rate of infected individuals through both mortality and recovery. If we denote by $\x(t)=(S(t),I(t))^{\sf T}$ the column vector that describes the state of the system at time $t$, we can rewrite system \eqref{SI} in the more compact form $\dot{\x}=f(\x)$.

In our analysis, we consider four different control policies, namely: vaccination, isolation, culling, and reduction of transmission. We denote the generic control policy rate applied at time $t$ by $u(t)$, which is assumed to be a piecewise continuous function that takes values in a positive bounded set $U=[0,u^{max}]$. We apply the different policies separately by adding a linear term in the control variable $u(t)$ to model \eqref{SI}, namely considering the general system
\begin{equation}
\dot{\x}(t)=f(\x(t))+u(t)g(\x(t)),
\label{dot_x}
\end{equation}
where the function $g$ depends on the chosen control policy. Specifically, we define a general linear term policy
\begin{equation}\label{gL}
g_{l}(\x)= \begin{pmatrix} -\alpha_1 S\\ -\alpha_2 I \end{pmatrix}
\end{equation}
which is a linear function of $S$ and $I$ that allows to model
\begin{align}
	\text{Vaccination}  									\qquad\qquad & \alpha_1=1, \alpha_2= 0: \qquad g_v(\x)= \begin{pmatrix} -S\\ 0 \end{pmatrix}\label{gv} \\
	\text{Isolation} 										\qquad\qquad & \alpha _1=0, \alpha_2= 1: \qquad g_i(\x)= \begin{pmatrix}  0\\-I \end{pmatrix}\label{gq} \\
	\text{Culling} 											\qquad\qquad &  \alpha _1=1, \alpha_2= 1: \qquad g_c(\x)= \begin{pmatrix} -S\\-I \end{pmatrix}\label{gc}
\end{align}
and, in addition, we consider the nonlinear term policy
\begin{equation} \label{gr}
	\text{Reduction of transmission} 		\qquad\qquad  g_r(\x)= \begin{pmatrix} -\beta SI\\-\beta SI \end{pmatrix}.
\end{equation}

We define the {\it basic reproduction number} for model \eqref{SI} as $\R_0=\beta S(0)/\mu$, which represents the average number of secondary infections produced by a single infected individual in a completely susceptible population in the absence of control \citep{anderson79}. In addition, for each policy we will define the {\it control reproduction number} $\R_C$ that represents the average number of secondary infections produced by a single infected individual in a completely susceptible population with control measures in place \citep{allen}.
From this definition, it follows that, when $\R_C<1$, control measures applied at the beginning of the epidemic are able to immediately reduce the number of the infected individuals (i.e. $\dot{I}(0)<0$).
The target will be the minimization of the eradication time of the infection. Existence of the eradication time in problem \eqref{dot_x} is guaranteed by the results in \ref{appendice}.
\begin{defin}[Eradication Time] The eradication time $T$ of the controlled SIR problem \eqref{dot_x} is the first time at which the number of infected individuals reaches the threshold $\varepsilon$, where $\varepsilon<1$ is a fixed positive constant. \end{defin}

We will chose initial conditions of infected individuals $I(0)$ strictly greater than $\varepsilon$. As a consequence,   $T$ being the first time at which the variable $I$ reaches $\varepsilon$, it holds that $\dot{I}(T)<0$.

%-----------------------------------------------------------------------------
We can then write the optimal control problem \citep{pontryagin} where the goal is:
\begin{align}
\text{minimize:} 		\quad & J(u) = \int_0^T 1 dt \quad\text{(Eradication time)}\notag\\
\text{subject to:} 	\quad & \dot{\x}(t)=f(\x(t))+u(t)g(\x(t)), \quad t\geq0;\label{TOC} \\
													& \x(0)=\x_0, \quad \x(T)\in \mathcal{C} = \{ (S,I)\,:\; I=\varepsilon \} \notag\\
													& u : [0,+\infty) \rightarrow U=[0,u^{max}] \text{ piecewise continuous},\notag
\end{align}
where $g$ is defined by the chosen control policy.

Given the optimal control problem \eqref{TOC} with $f,g\in\mathcal{C}^\infty(\mathbb{R}^2)$, we apply the Pontryagin's Minimum Principle \citep{pontryagin} in order to find a characterization of the optimal control strategy.

\begin{teo}[Pontryagin's Minimum Principle for linear time-optimal control problem \citep{pontryagin}]\label{PMP}
Suppose that $u^*(t)$ is a minimizer for the optimal control problem \eqref{TOC} and let $\x^*(t)=(S^*(t),I^*(t))^{\sf T}$ and $T^*$ denote the optimal solution of problem \eqref{dot_x} and the optimal eradication time, respectively. Then, there exists a piecewise $\mathcal{C}^1$ vector function $\llambda^*(t)=(\lambda^*_S(t),\lambda^*_I(t))^{\sf T}\neq0$ such that
\[ \dot{\llambda}^*(t) = -\nabla_x \H (\x^*(t),u^*(t),\llambda^*(t))^{\sf T}, \]
where the Hamiltonian is defined as $\H(\x,u,\llambda)=1+\llambda^{\sf T} (f(\x)+ug(\x))$, and:
\begin{enumerate}
	\item the function $h(w)=\H(\x^*(t),w,\llambda^*(t))$ attains its minimum on U at $w=u^*(t)$:
	\[ \H(\x^*(t),u^*(t),\llambda^*(t)) \leq \H(\x^*(t),w,\llambda^*(t)), \quad \forall w\in U\]
	for every $t\in[0,T^*]$;
	\item the Hamiltonian is constant equal to zero along the optimal solution: \[ \H(\x^*(t),u^*(t),\llambda^*(t))=0; \]
	\item the following transversality condition holds: $\lambda^*_S(T^*)=0$.
\end{enumerate}
Moreover, because the Hamiltonian is linear in the control variable, the value of $u^*(t)$ is determined by the sign of the switching function $\psi(\x,\llambda)=\llambda^Tg(\x)$ for all the time instants $t$ at which $\psi(\x^*(t),\llambda^*(t))$ does not vanish:
\[ u^*(t)= \begin{cases}
			0 & \text{if } \psi(\x^*(t),\llambda^*(t))>0\\
			u^{max} & \text{if } \psi(\x^*(t),\llambda^*(t))<0. \end{cases} \]
\end{teo}

\subsection{Admissible optimal controls and numerical method}

The results that we will prove in the next sections can be summarized in the following theorem.

\begin{teo}\label{teo bang-bang} For each considered control policy the optimal control for problem \eqref{TOC} is bang-bang. The optimal strategy consists either in a constant control $u^*(t)\equiv u^{max}$ or in a delayed control $0\rightarrow u^{max}$ with a single switching time $\tau_s^*$, namely $u^*(t)=0$ for $t\in[0,\tau_s^*)$ and $u^*(t)=u^{max}$ for $t\in(\tau_s^*,T^*]$. In addition, if the optimal control is delayed, three different behaviors are allowed, depending on the position of the switching time $\tau_s^*$ compared to the peak of infection, leading to the four different types of admissible optimal control sketched in Fig. \ref{fig:legenda}.
\end{teo}

\begin{figure}\centering
\includegraphics[width=\columnwidth]{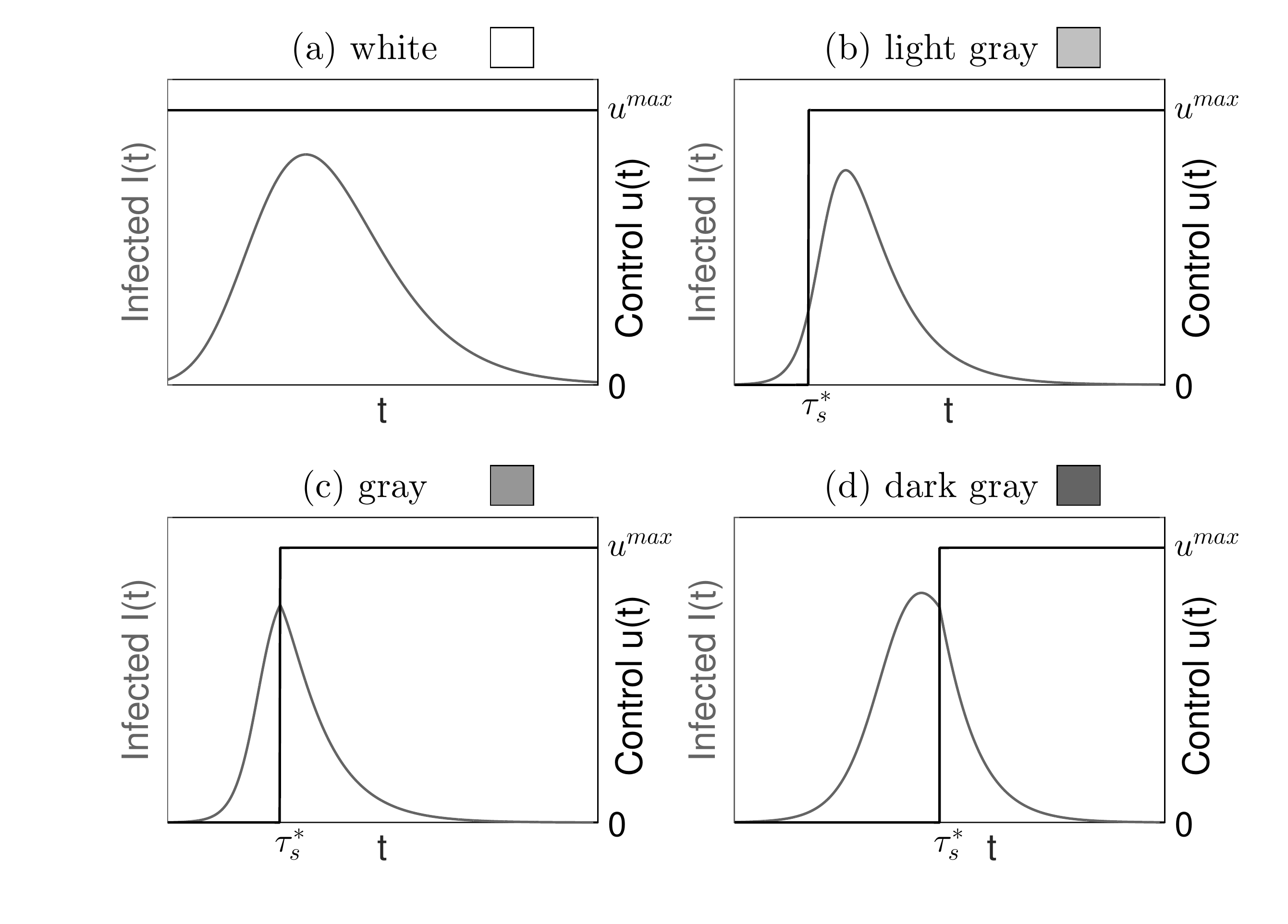}
\caption{Schematization of the four types of admissible optimal control and legend of the plot colors that will be used throughout the article. White (panel a) denotes a constant control at its maximum value. Different shades of gray denote delayed controls applied at the switching time $\tau^*_s$. We distinguish three different behaviors, depending on the position of $\tau^*_s$ with respect to the infectious dynamics: the switch occurs before the peak of the infection (panel b, light gray), in correspondence of the peak (panel c, gray) or after the peak of infection (panel d, dark gray).}\label{fig:legenda}
\end{figure}

We will denote the set of such admissible optimal controls by
\begin{align}
\A=\bigg\{ u:&[0,+\infty)\rightarrow\{0,\,u^{max}\} \text{ piecewise constant with at most}\notag\\
			 &\text{ one jump from 0 to $u^{max}$ }, \lim_{t\rightarrow+\infty}u(t)=u^{max} \bigg\}.\label{A}
\end{align}

As regards the numerical solution, several numerical methods for the optimal solution of both minimum time and bang-bang control problems can be found in literature. Such techniques are mainly based on shooting methods \citep{kaya,lenhart,martinon}, smooth regularizations of the control function \citep{silva}, or pseudospectral methods \citep{shamsi}. However, since our problem is characterized by the particular class of optimal controls $\A$ in \eqref{A}, for our numerical simulations we will use a simpler ad hoc numerical scheme. The method is based on the idea of identifying each bang-bang function $u(t)\in\A$ with a real parameter. Since the delayed optimal control function is not defined at the switching time instant, for numerical simulations we fix by convention $u^*(\tau_s^*)=u^{max}$. Then, we can generalize the idea of switching time, defined as the zero of the switching function $\psi$, introducing the {\it \interv time}
\begin{equation}
\tau= \begin{cases}
0 			\quad & \text{in case of constant maximum control}\\
\tau_s 	\quad & \text{in case of delayed control}
\end{cases}
\label{eq:tau}
\end{equation}
which represents the first time instant at which the control $u(t)\in\A$ assumes the value $u^{max}$. Since a delayed control can be characterized by its switching time $\tau_s$, we can then identify each admissible optimal control in $\A$, constant or delayed, by the value of $\tau$ and write it in the more general form:
\begin{equation}
u(t;\tau)= \begin{cases}
0 				\quad & 0<t<\tau\\
u^{max} 	\quad & \tau\leq t<+\infty.
\end{cases}
\label{eq:utau}
\end{equation}

Therefore, the functional to be minimized $J(u)$ can be seen as the function\\ $J:\tau\rightarrow T$, that links the \interv time $\tau\geq0$ to the eradication time $T$ of the problem
\begin{equation}
\begin{aligned}
\dot{\x}(t)=f(\x(t))+u(t;\tau)g(\x(t)), \quad t\geq0; \\
\x(0)=\x_0, \quad \x(T)\in \mathcal{C} = \{ (S,I)\,:\; I=\varepsilon \}.
\end{aligned}
\label{J(tau)}
\end{equation}
An optimal control $u^*$ will be identified by a \interv time such that $\tau^*= \text{argmin}\, J$, since it can be proved that $J$ always admits at least a minimum value (see Theorem \ref{teo:exist}).
The numerical solution will be computed by evaluating the function $J(\tau)$ over a suitable interval and looking for its minimum value. In particular, we fix a uniform mesh $\{\tau_i,\, i=1,\ldots,M\}$ over the interval $[0,T_{unc}]$, where $T_{unc}$ is the eradication time of the uncontrolled epidemic. For each mesh point we consider the related control function $u(t;\tau_i)$ and numerically integrate the Cauchy problem \eqref{J(tau)} using the Crank-Nicholson method with uniform time steps $\{t_k,\, k=1,\ldots,N\}$, obtaining the numerical solution $\x^{(i)}_k=(S^{(i)}_k,I^{(i)}_k)^{\sf T}$, $k=1,\ldots,N$. Then, we set the eradication time $T_i$ relevant to the mesh point $\tau_i$ as the first time step $t_{\bar{k}}$ at which the computed solution $I^{(i)}_{\bar{k}}\leq\varepsilon$. Finally, we take the minimum over the set of computed eradication times $T_j=\min\{T_i,\, i=1,\ldots,M\}$ as the optimal eradication time, and set the corresponding $\tau_j$ as the optimal starting intervention time.

In the following sections we investigate the four different control policies considered. For each policy we will present theoretical and numerical results.  In our numerical simulations, we set $\varepsilon = 0.5$, as in \citep{hansen11JOMB}. Then, through a sensitivity analysis, we explore the solutions of optimal control problems \eqref{gv}--\eqref{gr} on a wide range of parameter settings describing different epidemiological conditions (represented by $\R_0 = \beta S(0)/ \mu$), different possible control efforts (represented by $u^{max}$), and a different number of initially introduced infected individuals in the population (represented by $I(0)$).

\section{Linear term policies}
We consider SIR model \eqref{SI} with the general linear term control, denoted by $\uu(t)$, obtaining an optimal control problem as the one defined in \eqref{TOC}, with $g_{l}(\x)$ as in \eqref{gL}.

\begin{teo}\label{teo:VIC} If $\uu^*$ is the optimal control strategy for the linear term control problem, then $\uu^*$ is a bang-bang control with at most one switching time $\tau_s^*$ from no control to maximum control.
\end{teo}
\begin{proof} See \ref{dimgen}. \end{proof}

We proceed now to analyze the peculiarities of each policy involved in the general formulation.

\subsection{Vaccination}
We consider the vaccination control, denoted by $\uv(t)$ in the optimal control problem  \eqref{TOC}, with $g_v(\x)$ as in \eqref{gv}. The control reproduction number for vaccination is defined as $\R_C^v=\R_0=\beta S(0)/\mu$.

For this policy, it is easy to prove that there exists a unique time instant $t_p$ (possibly 0) at which the function $\dot{I}$ changes sign. In particular, $\dot{I}(t)>0$ for $t<t_p$ and $\dot{I}(t)<0$ for $t>t_p$. We call $t_p$ the peak time, because it represents the time at which the number of infected individuals reaches its maximum. Therefore, in addition to the general results of Theorem \ref{teo:VIC}, it is possible to prove the following.

\begin{teo}\label{teo:vacc}  The switch of the optimal control $\uv^*$ can occur only before the peak of the infection. Moreover, if $\R_0<1$ or $\vmax>\mu$, the optimal control is the constant control $\uv^*(t)\equiv \vmax$.
\end{teo}
\begin{proof} See \ref{propvacc}. \end{proof}

\begin{figure}
\centering
\includegraphics[width=.8\columnwidth]{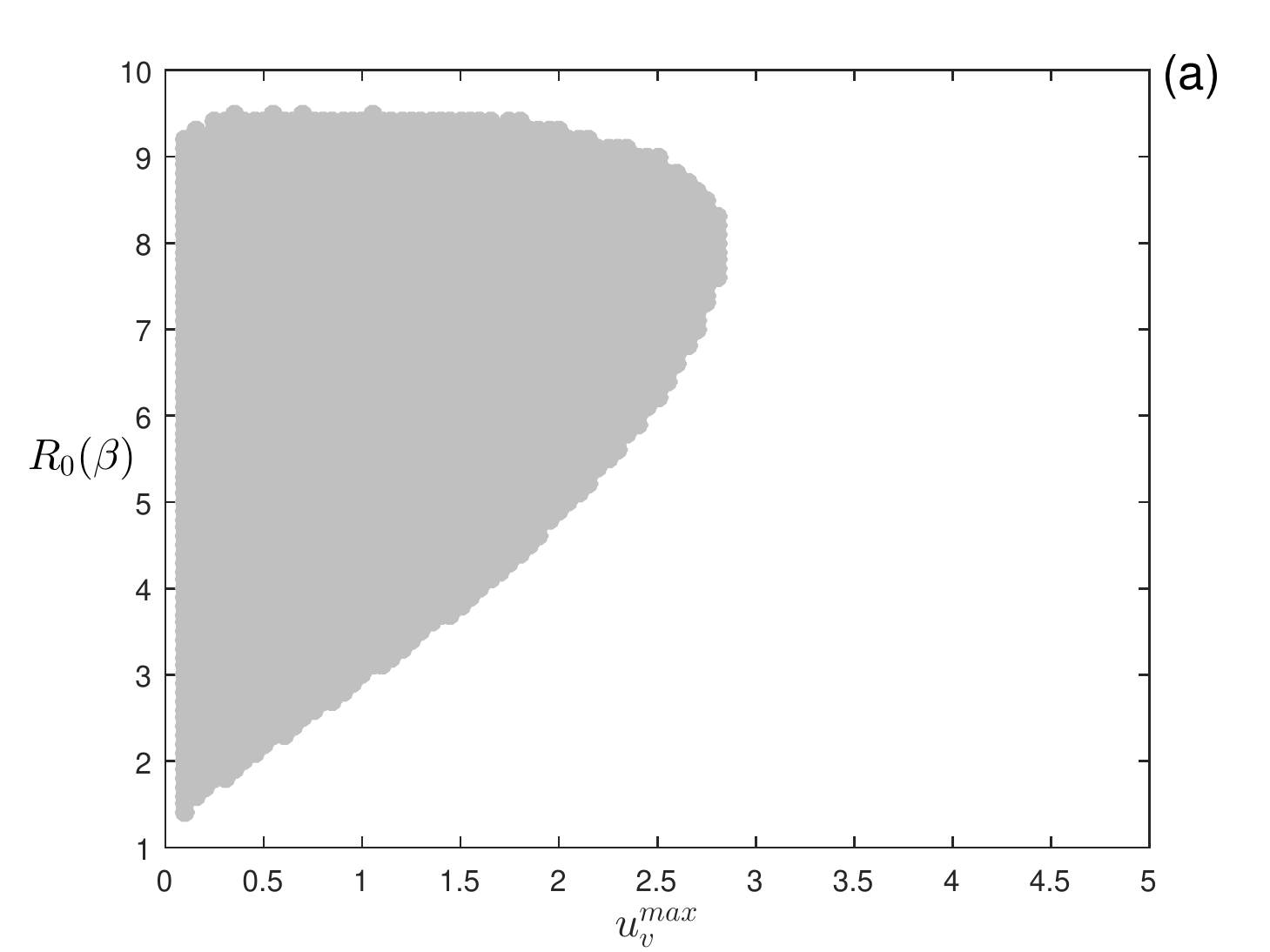}
\includegraphics[width=.8\columnwidth]{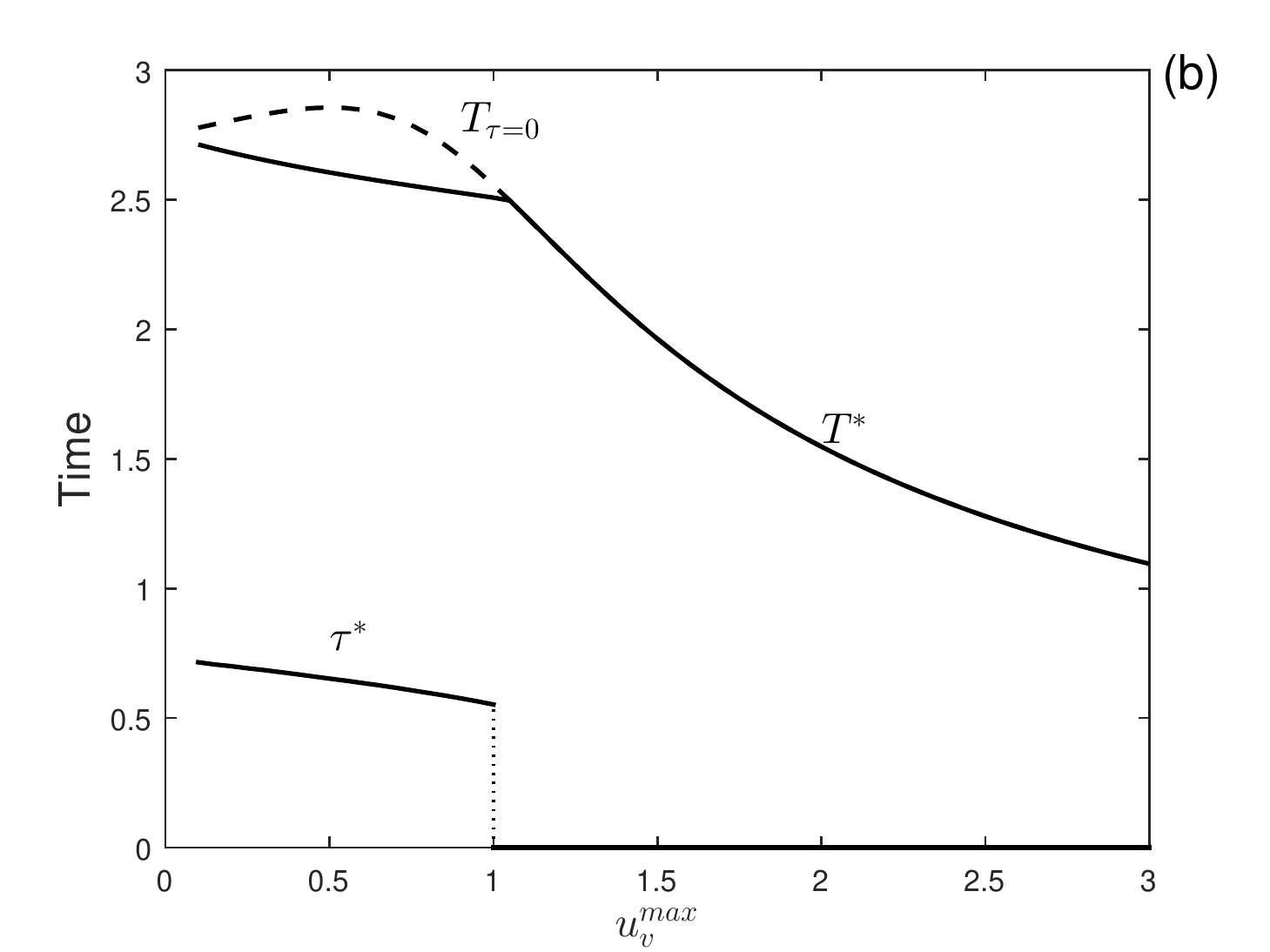}
\caption{Numerical analysis of the optimal vaccination problem. (a) Different colors represent different optimal control types obtained by varying $\vmax$ (that ranges from 0 to $\mu$) and $\R_0(\beta)$. Color meanings are specified in Fig. \ref{fig:legenda}. (b) Plot of the optimal \interv time $\tau^*$, the optimal eradication time $T^*$, and the eradication time $\Tumax$ as functions of $\vmax$, with $\R_0(\beta)=3$. Other parameters: $S(0)=2000$, $\mu=5$, $I(0)=1$, $\varepsilon=0.5$. }
\label{fig:v}
\end{figure}

The numerical analyses on the time-optimal vaccination problem are illustrated in Fig. \ref{fig:v}. In Fig. \ref{fig:v}(a), we show the results of the simulations performed in the parameter space [$\vmax$, $\R_0(\beta)$]. As explained in the color codes in Fig. \ref{fig:legenda}, the light gray and white regions in Fig. \ref{fig:v}(a) represent the combinations of parameters [$\vmax$, $\R_0(\beta)$] for which the time-optimal vaccination problem selects for delayed and constant control, respectively. As highlighted by the analytic results, Fig. \ref{fig:v}(a) displays that the switching time always occurs before the peak of infection and that higher vaccination efforts always select for a constant control. Fig. \ref{fig:v}(b) shows the optimal \interv time ($\tau^*$), the eradication time for the optimal vaccination strategy ($T^*$, solid curve), and the eradication time for the constant vaccination ($\Tumax$, dashed curve) as functions of the maximum effort, $\vmax$. In Fig. \ref{fig:v}(b), we notice that the optimal \interv time undergoes a ``catastrophic'' transition (\textit{sensu} \citep{thom72}) from delayed to constant control for increasing values of $\vmax$. Then, small changes in $\vmax$ can cause an abrupt change in the starting point of the optimal vaccination campaign. On the other hand, Fig. \ref{fig:v}(b) shows that, when delaying the onset of vaccination is optimal, the differences in the final time of the epidemic between optimal control and constant control (i.e. variation in the objective function) is marginal. 

\subsection{Isolation}
We consider SIR model \eqref{SI} with isolation control, denoted by $\uq(t)$, obtaining an optimal control problem as the one defined in \eqref{TOC}, with $g_i(\x)$ as in \eqref{gq}.
The control reproduction number for isolation is defined as $\R_C^i=\beta S(0)/(\mu+\qmax)$.

The numerical analyses on the time-optimal isolation problem are illustrated in Fig. \ref{fig:q}.
In Fig. \ref{fig:q}(a), we show the results of the simulations performed in the parameter space [$\qmax$, $\R_0(\beta)$]. Conversely to vaccination, our results show that the time-optimal isolation problem can select for delayed strategies also for high values of maximum effort, $\qmax$. Moreover, the switching time for the optimal isolation strategy can occur after the peak of infection (see the dark gray region in Fig. \ref{fig:q}(a)). In Fig. \ref{fig:q}(b), we show that the isolation problem selects for optimal delayed control in a wide range of parameter settings also when the number of infected individuals firstly introduced in the population increases (i.e. $I(0)>1$). Fig. \ref{fig:q}(c) displays the optimal \interv time ($\tau^*$), the final time for the optimal isolation strategy ($T^*$, solid curve), and the final time for the constant isolation ($\Tumax$, dashed curve) as functions of the maximum effort, $\qmax$. As in the vaccination problem, the optimal \interv time for isolation undergoes a ``catastrophic'' transition from delayed to constant control for increasing values of maximum effort. Fig. \ref{fig:q}(c) shows that delayed control can be optimal also when $\R_C < 1$, i.e. when an prompt intervention at $t = 0$ could have implied an immediate decline in the number of infected individuals.
In addition, when delaying the onset of isolation is optimal, the differences in the final time of the epidemic between optimal control and constant control can be significant. Fig. \ref{fig:q}(d) shows the number of susceptible individuals at the end of the epidemic for the optimal isolation strategy ($S(T^*)$, solid curve) and the constant isolation ($S(\Tumax)$, dashed curve) as functions of the maximum effort, $\qmax$. Similarly to the switching time, $S(T^*)$ exhibits a discontinuous increase at the boundary between delayed and constant control.

\begin{sidewaysfigure}
\centering
\includegraphics[width=.4\columnwidth]{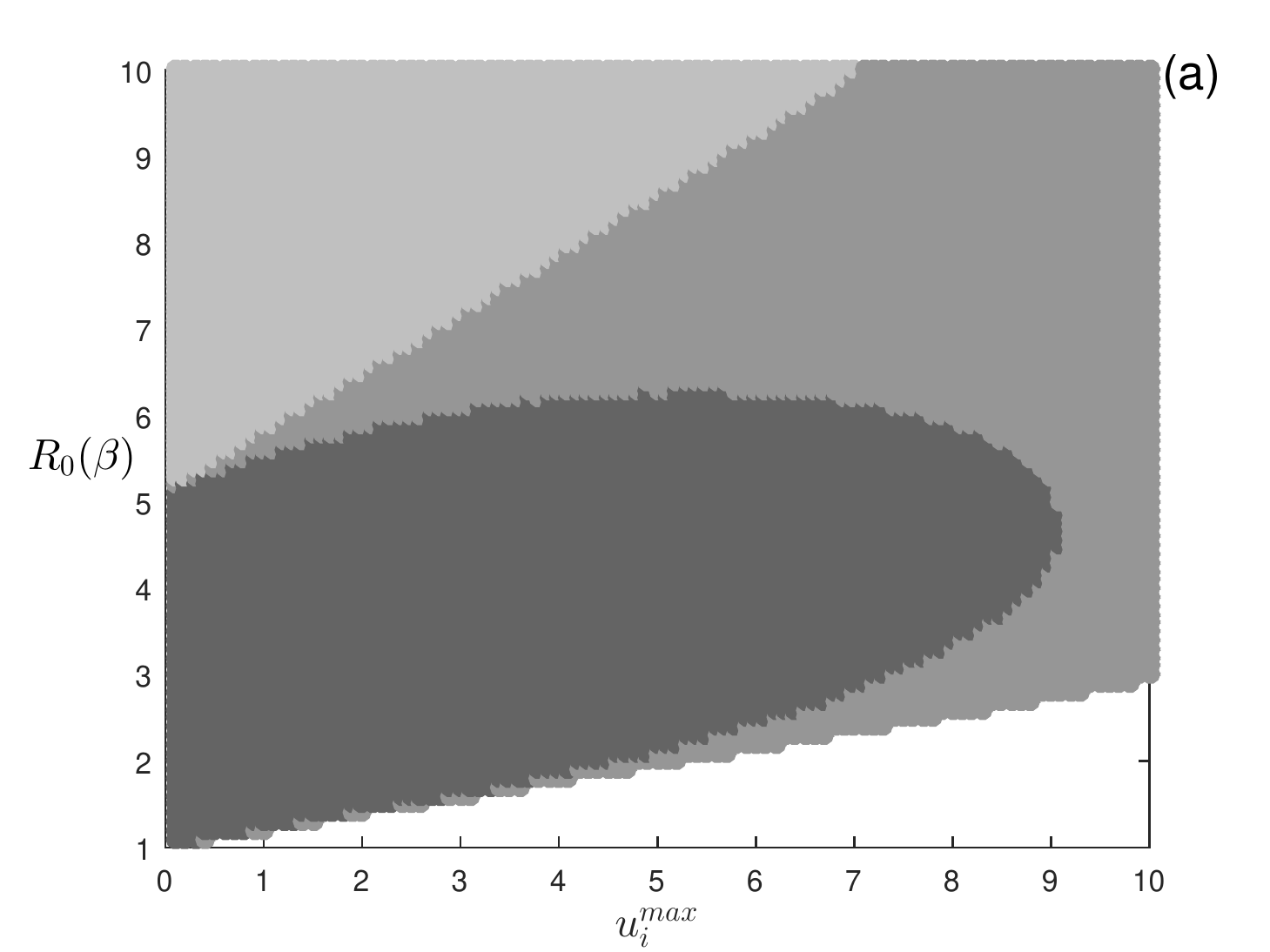}
\includegraphics[width=.4\columnwidth]{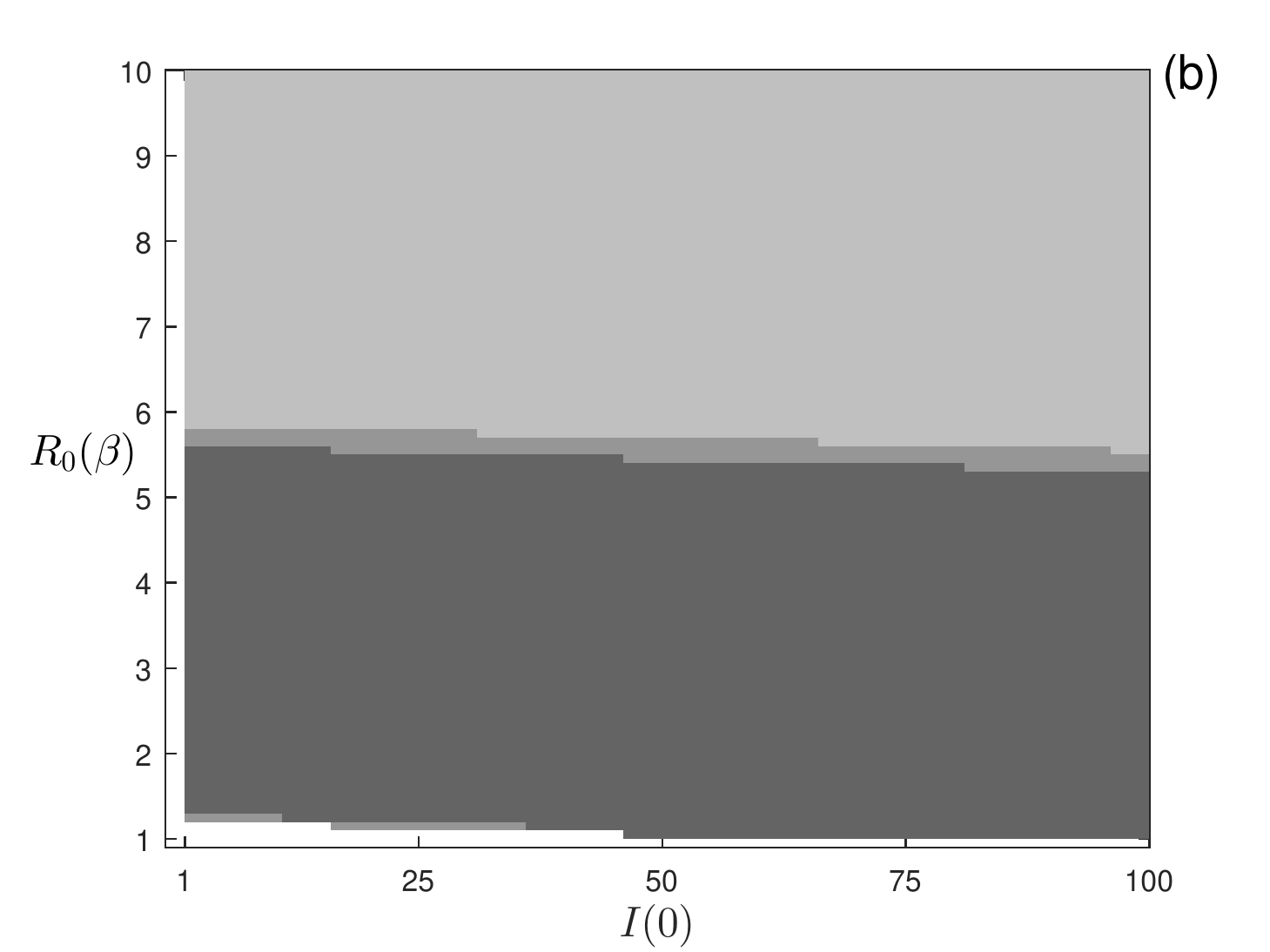}\\
\includegraphics[width=.4\columnwidth]{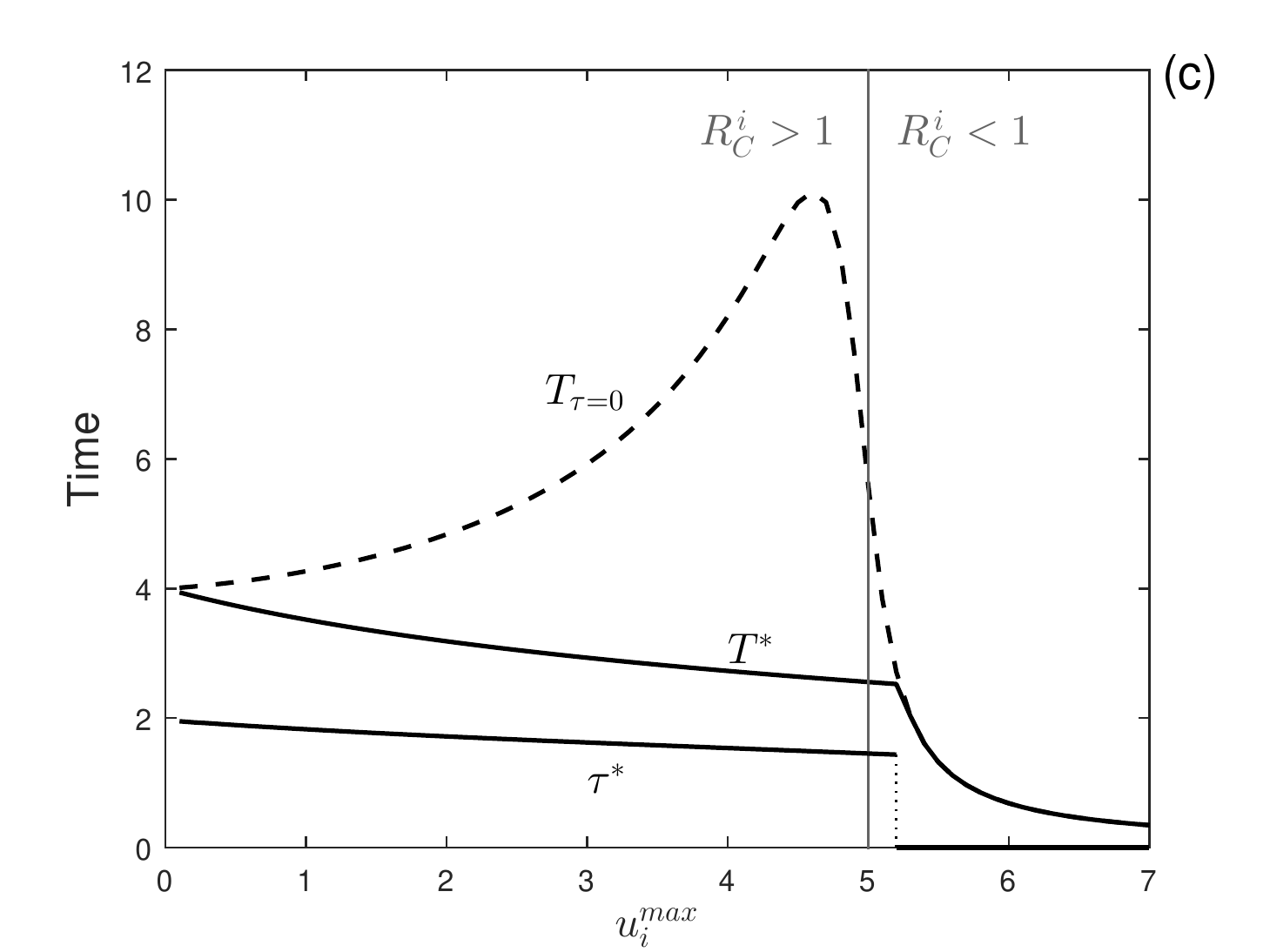}
\includegraphics[width=.4\columnwidth]{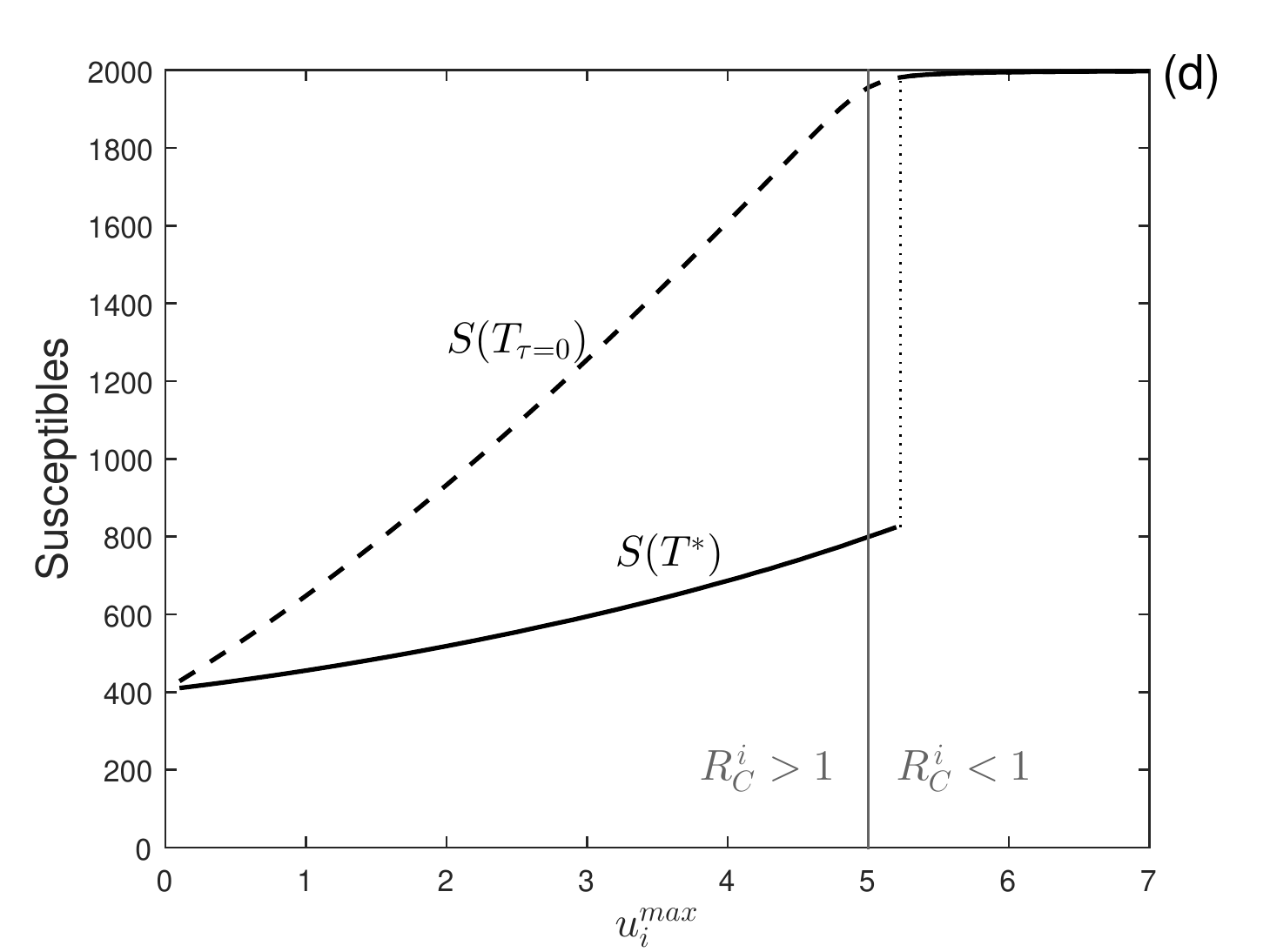}
\caption{Numerical analysis of the optimal isolation problem. Different colors represent different optimal control types obtained by varying: (a) $\qmax$ and $\R_0(\beta)$; (b) $I(0)$ and $\R_0(\beta)$. Color meanings are specified in Fig. \ref{fig:legenda}. (c) Plot of the optimal \interv time $\tau^*$, the optimal eradication time $T^*$, and the eradication time $\Tumax$ as functions of $\qmax$, with $\R_0(\beta)=2$. (d) Number of susceptible individuals at the end of the epidemic obtained using the optimal control, $S(T^*)$, and the constant control $u(t)=\qmax$, $S(\Tumax)$, as functions of $\qmax$, with $\R_0(\beta)=2$. In panel (b) $\qmax = 1$. In panels (c)-(d) is also highlighted the value of $\qmax$ for which $\R^i_C=1$ (in gray). Other parameter values as in Fig. \ref{fig:v}.}
\label{fig:q}
\end{sidewaysfigure} 

\subsection{Culling}
We consider the culling control, denoted by $\uc(t)$, in the optimal control problem defined in \eqref{TOC}, with $g_c(\x)$ as in \eqref{gc}. The control reproduction number for culling is defined as $\R_C^c=\beta S(0)/(\mu+\cmax)$.

\begin{figure}
\centering
\includegraphics[width=.7\columnwidth]{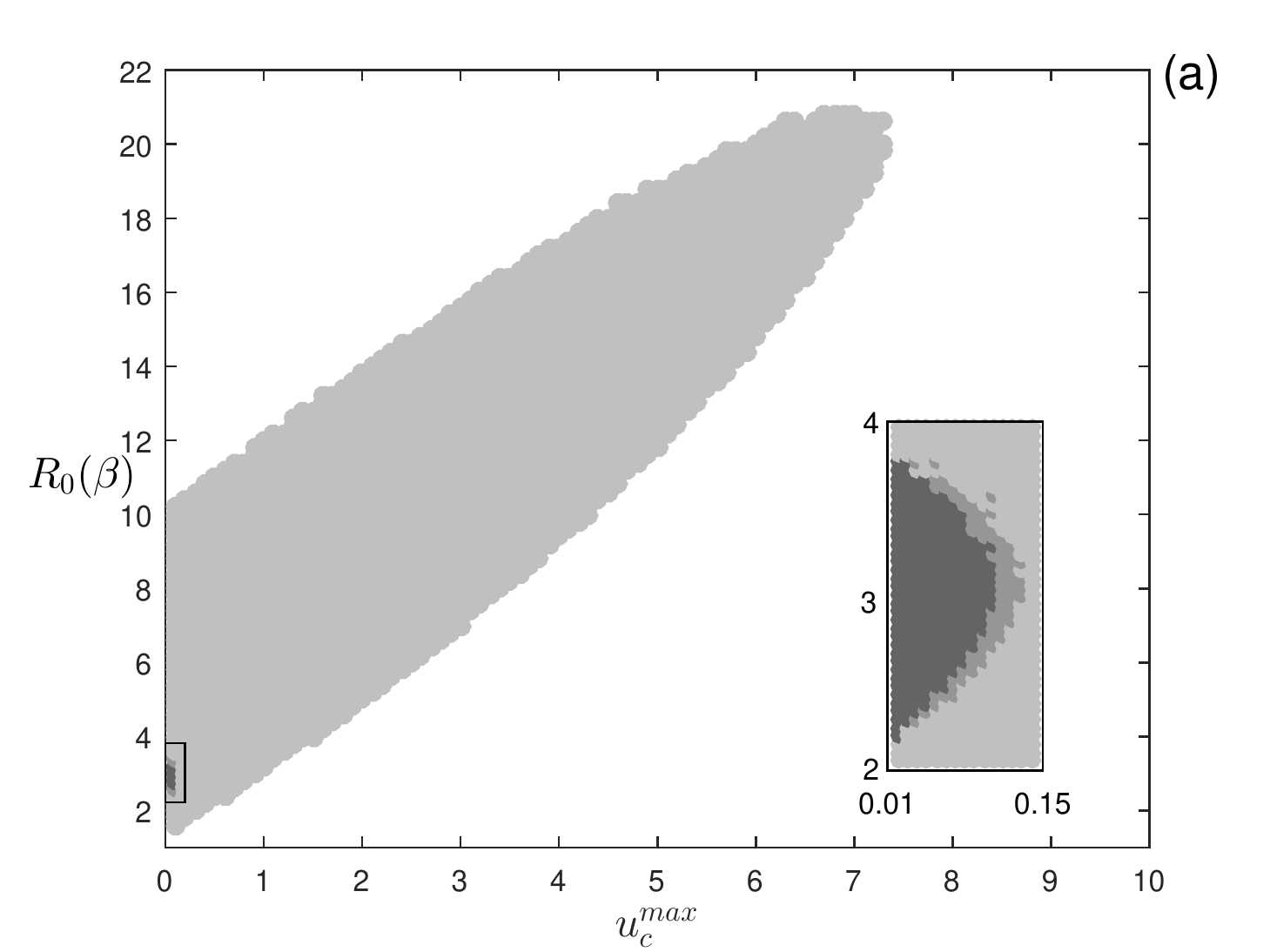}\\
\includegraphics[width=.7\columnwidth]{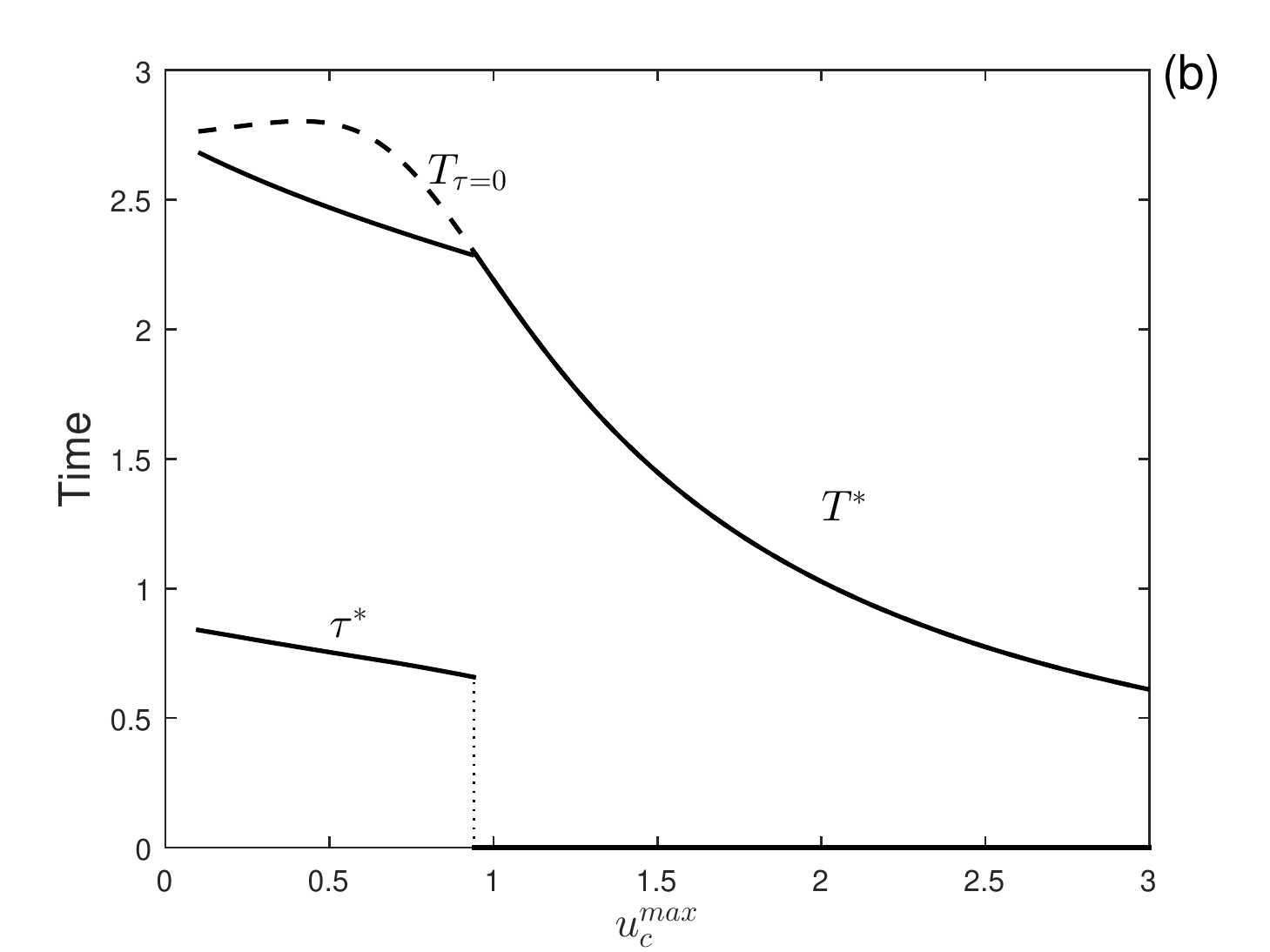}
\caption{Numerical analysis of the optimal culling problem.  (a) Different colors represent different optimal control types obtained by varying $\cmax$ and $\R_0(\beta)$. Color meanings are specified in Fig. \ref{fig:legenda}. (b) Plot of the optimal \interv time $\tau^*$, the optimal eradication time $T^*$, and the eradication time $\Tumax$ as functions of $\cmax$, with $\R_0(\beta)=3$. Other Parameter values as in Fig. \ref{fig:v}.}
\label{fig:c}
\end{figure}

The numerical analyses on the time-optimal culling problem are illustrated in Fig. \ref{fig:c}. In Fig. \ref{fig:c}(a), we show the results of the simulations performed in the parameter space [$\cmax$, $\R_0(\beta)$]. We display that, when $\R_0$ is low, delayed control is selected for small values of culling effort ($\cmax$), while, when $\R_0$ is high, delayed control is selected for intermediate values of $\cmax$. In addition, in the aforementioned cases, the starting of the optimal culling generally occurs before the peak of infection (light gray region in Fig. \ref{fig:c}(a)). However, we can notice that there exists a small region in the parameter space [$\cmax$, $\R_0(\beta)$] where the starting of the optimal strategy can occur after the peak of infection (see the dark gray region in the box). Fig. \ref{fig:c}(b) shows the optimal \interv time ($\tau^*$), the final time for the optimal culling strategy ($T^*$, solid curve), and the final time for the constant culling ($\Tumax$, dashed curve) as functions of the maximum effort, $\cmax$. Also in this case the optimal \interv  time undergoes a ``catastrophic'' transition from delayed to constant control for increasing values of $\cmax$ and, when delaying the onset of culling is optimal, the differences in the final time between optimal control and constant control is marginal, analogously to the case of vaccination. 

\section{Reduction of transmission policy}
We consider SIR model \eqref{SI} with reduction of transmission control, denoted by $\ur(t)$, obtaining an optimal control problem as the one defined in \eqref{TOC}, with $g_r(\x)$ as in \eqref{gr} and $0<\rmax\leq1$. The control reproduction number for reduction of transmission is defined as $\R_C^r= \beta(1-\rmax)/\mu$. Despite the nonlinearity of this kind of policy, it is possible to find the same type of optimal strategy of the linear term policies.

\begin{teo}  If $\ur^*$ is the optimal control strategy for the reduction of transmission problem, then $\ur^*$ is a bang-bang control with at most one switching time $\tau_s^*$ from no control to maximum control.
\label{teo:rid}\end{teo}
\begin{proof} See \ref{dimR}. \end{proof}

\begin{sidewaysfigure}
\centering
\includegraphics[width=.4\columnwidth]{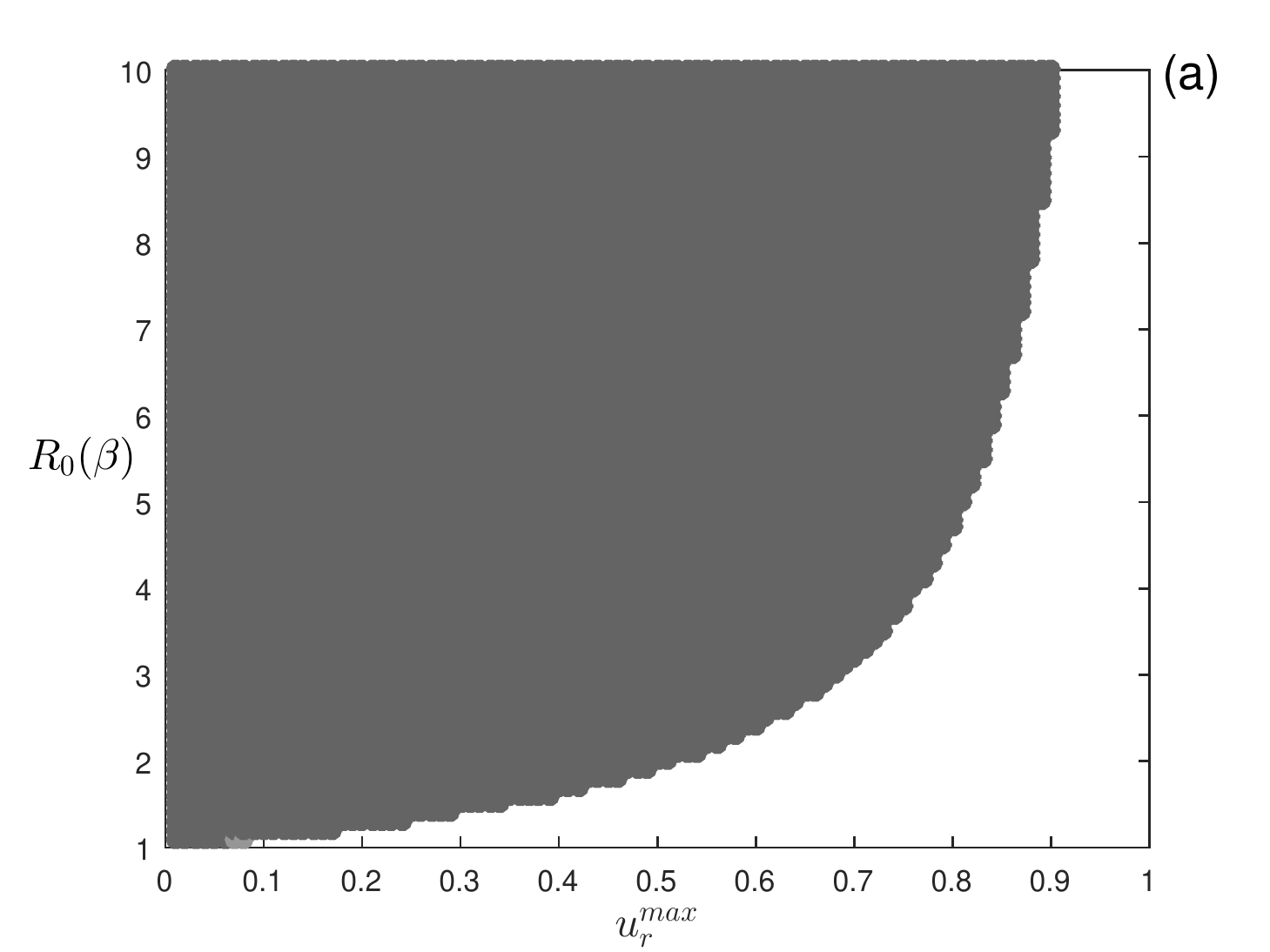}
\includegraphics[width=.4\columnwidth]{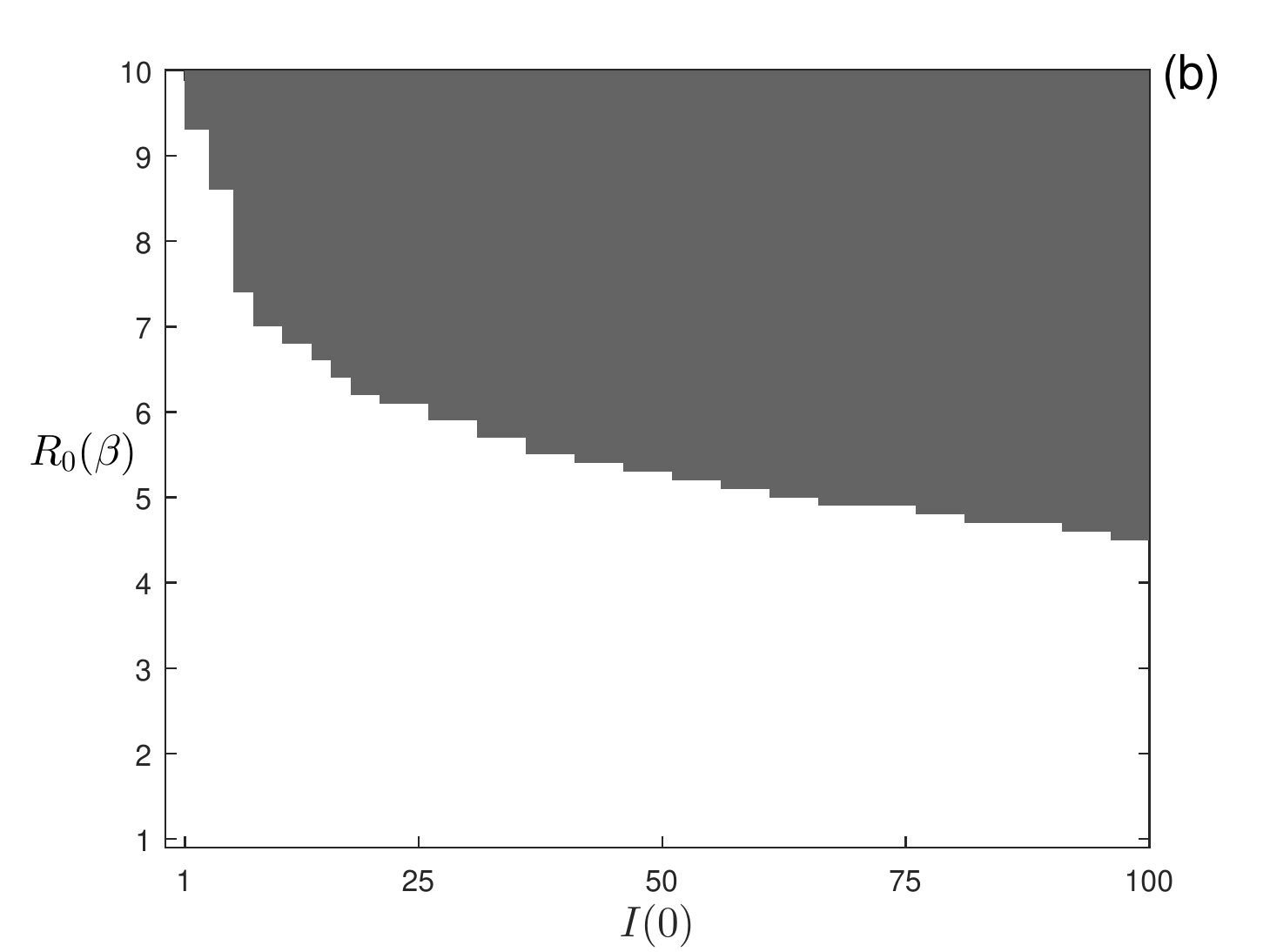}\\
\includegraphics[width=.4\columnwidth]{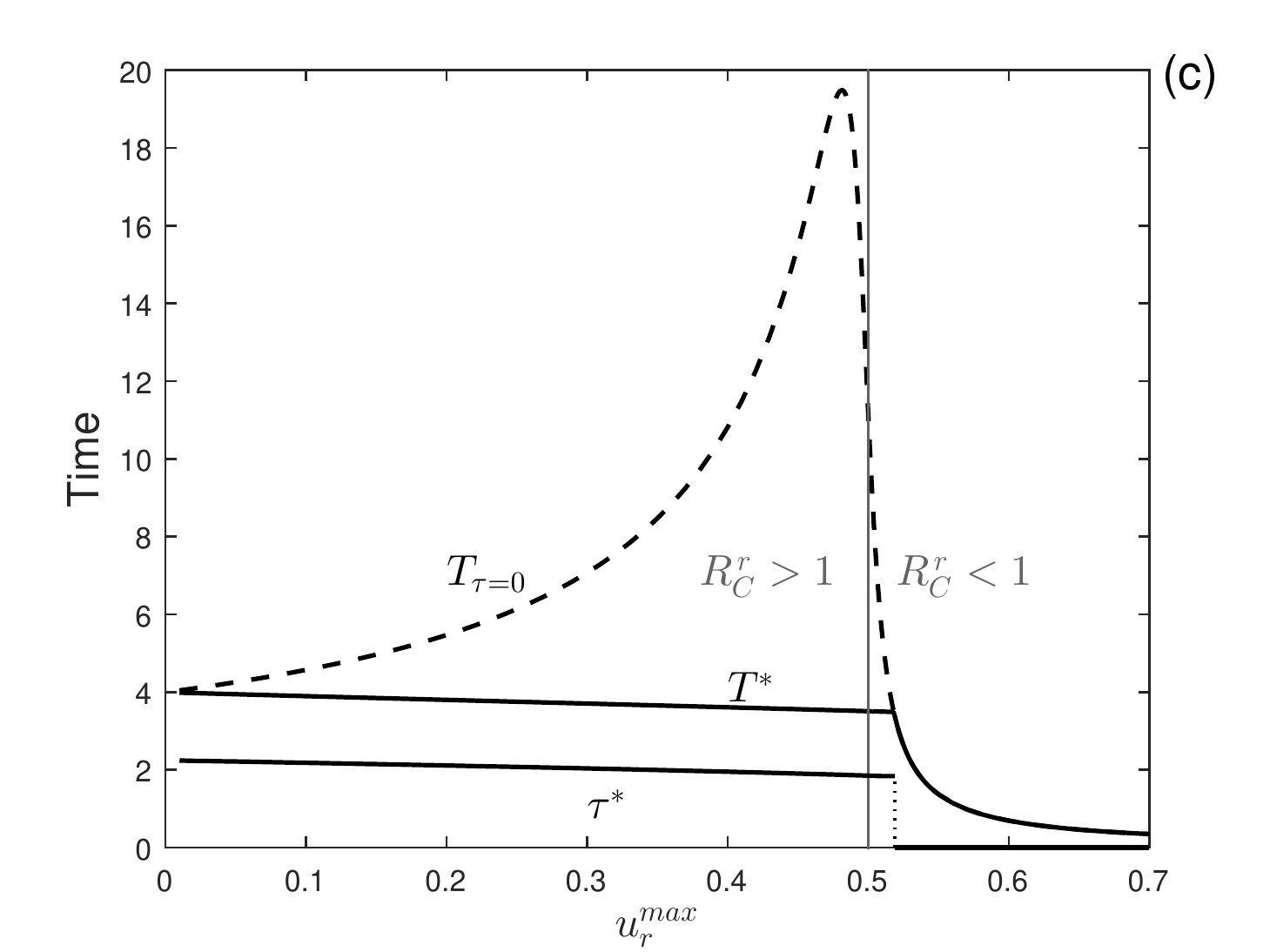}
\includegraphics[width=.4\columnwidth]{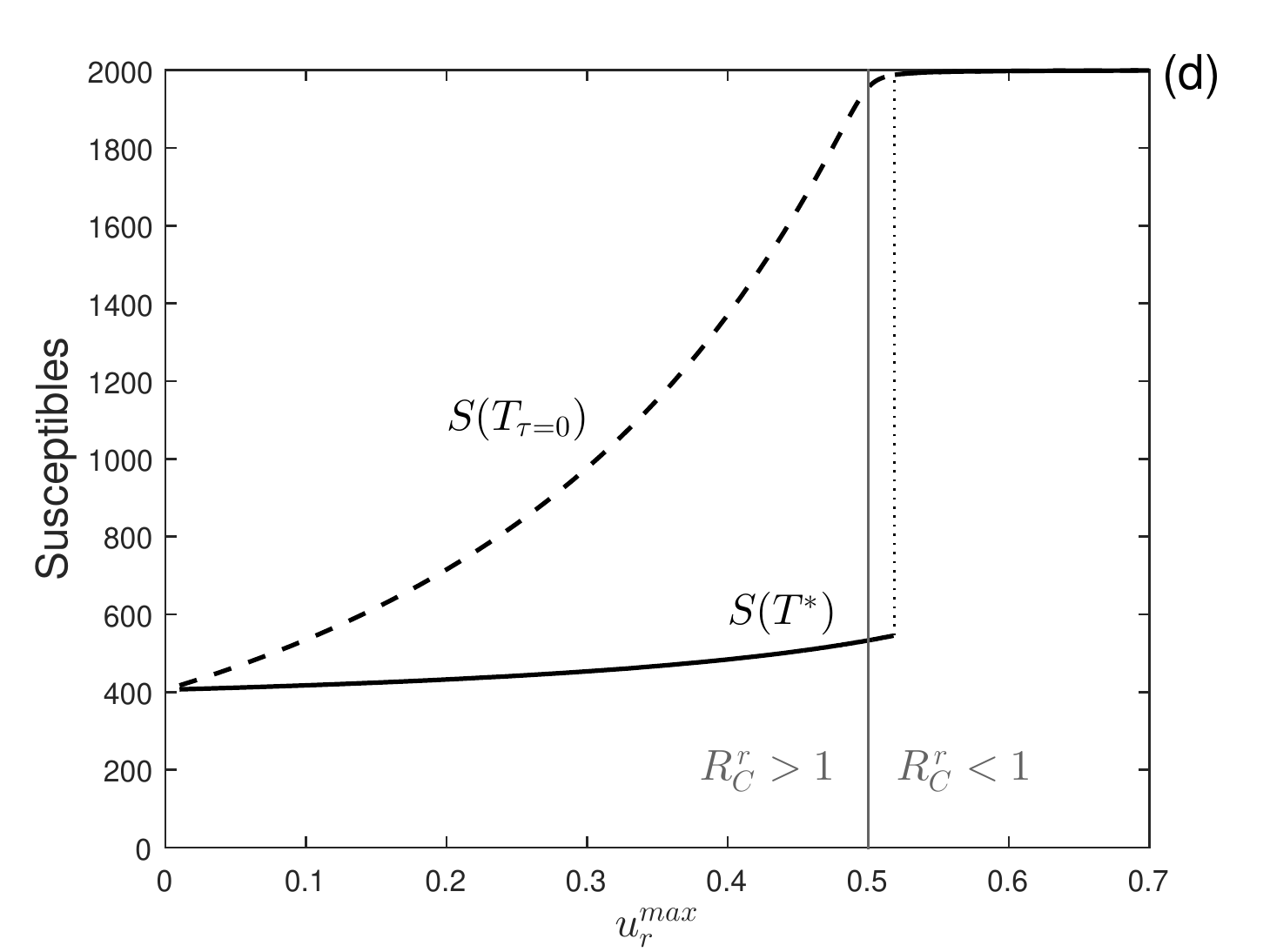}
\caption{Numerical analysis of the optimal reduction of transmission problem. Different colors represent different optimal control types obtained by varying: (a) $\rmax$ and $\R_0(\beta)$; (b) $I(0)$ and $\R_0(\beta)$. Color meanings are specified in Fig. \ref{fig:legenda}. (c) Plot of the optimal \interv time $\tau^*$, the optimal eradication time $T^*$, and the eradication time $\Tumax$ as functions of $\rmax$, with $\R_0(\beta)=2$. (d) Number of susceptible individuals at the end of the epidemic obtained using the optimal control, $S(T^*)$, and the constant control $u(t)=\rmax$, $S(\Tumax)$, as functions of $\qmax$, with $\R_0(\beta)=2$. In panel (b) $\rmax = 0.9$. In panels (c)-(d) is also highlighted the value of $\rmax$ for which $\R^r_C$=1 (in gray). Other parameter values as in Fig. \ref{fig:v}.}
\label{fig:r}
\end{sidewaysfigure}

The numerical analyses on the time-optimal reduction of transmission problem are illustrated in Fig. \ref{fig:r}. In Fig. \ref{fig:r}(a), we show the results of the simulations performed in the parameter space [$\rmax$, $\R_0(\beta)$]. We display that, when delayed control is selected, the starting of the optimal reduction of transmission generally occurs after the peak of infection (dark gray region in Fig. \ref{fig:r}(a)). In Fig. \ref{fig:r}(b), we show that the reduction of transmission problem selects for optimal delayed control in a wide range of parameter settings also when the number of infected individuals firstly introduced in the population increases (i.e. $I(0)>1$). Fig. \ref{fig:r}(c) shows the optimal \interv time ($\tau^*$), the final time for the optimal reduction of transmission strategy ($T^*$, solid curve), and the final time for the constant reduction of transmission ($\Tumax$, dashed curve) as functions of the maximum effort, $\rmax$. Similarly to the isolation problem, we find that: (\textit{i}) delayed control for reduction of transmission can be optimal also when $\R_C < 1$; and (\textit{ii}) when delaying the starting of reduction of transmission is optimal, the differences in the final time of the epidemic between optimal control and constant control can be significant. Fig. \ref{fig:r}(d) shows the number of susceptible individuals at the end of the epidemic for the optimal reduction of transmission strategy ($S(T^*)$, solid curve) and the constant reduction of transmission ($S(\Tumax)$, dashed curve) as functions of the maximum effort, $\rmax$. Similarly to the isolation problem, $S(T^*)$ exhibits a discontinuous increase at the boundary between delayed and constant control. 

\section{Discussion and conclusions}
In this work, we investigated the problem of minimizing the epidemic duration by using different control policies. Specifically, we characterized analytically the time-optimal control strategies for prophylactic vaccination, isolation, non-selective culling, and reduction of transmission by using a family of simple SIR models in an optimal control framework \citep{pontryagin}. Our analyses led to the non-trivial result that, even in the unconstrained optimal control problem (i.e. without costs of control or resource limitations), using the maximal effort for the entire epidemic period may not be the optimal strategy to minimize the epidemic duration. In addition, we found that, when applying the maximal effort for the entire epidemic is sub-optimal, then a delayed control represents the optimal strategy in all the cases investigated. We even found that the optimal amount of delay applied to the control may be sufficiently large to postpone the beginning of the intervention after the peak of the infection (see Fig. \ref{fig:legenda} and dark gray regions in Figs. \ref{fig:q}(a), \ref{fig:c}(a), and \ref{fig:r}(a)). In addition, we showed that the delayed control may represent the optimal strategy for minimizing the epidemic duration even when a prompt intervention could immediately reduce the number of infected individuals (i.e. reduce $\R_C$ below 1, see Figs. \ref{fig:q}(c) and \ref{fig:r}(c)).

The biological explanation for the optimality of delayed controls relies on the remark that, at the beginning of the epidemic, the infection process can be more efficient in depleting the reservoir of susceptibles (which represents the mechanism leading to epidemic extinctions) than the applied control.
In other words, reducing via external interventions the number of individuals involved in the infection process at the beginning of the outbreak (especially the infected ones) may lead to slower epidemic dynamics, which implies longer times for the epidemic to go extinct.
Two evidences support this explanation: (1) delayed control is generally optimal when the effectiveness of the control is low (i.e. low $u^{max}$); and (2) isolation and reduction of transmission policies (which do not reduce directly the number of susceptibles) tend to select for delayed control in wider ranges of parameter settings than vaccination and culling.

Our results differ from those previously obtained for the time-optimal problem in specific epidemic contexts. By analysing a subsystem of an epidemic model describing SARS spread, Jiang \citep{jiang07} proved that, according to Pontryagin's Minimum Principle, maximizing the isolation effort for the entire epidemic period would reduce epidemics in minimum time. Similarly, by numerically testing scenarios in an SIR model where the control always reduces the disease reproduction number below 1, Iacoviello \& Liuzzi \citep{iacoviello08} showed that maximizing the combined vaccination and isolation efforts for the entire epidemic period eradicates epidemics in minimum time.

Our results substantially differ also from those obtained minimizing the total number of infected (or the infectious burden) in SIR epidemic models. By characterizing optimal controls according to Pontryagin's Minimum Principle, different works showed that the unconstrained problems for isolation \citep{wickwire75, hansen11JOMB}, vaccination \citep{morton74, hansen11JOMB}, and culling \citep{bolzoni14} only support the trivial solution of applying the maximal effort for the entire epidemic. Then, from our results it follows that the infectious burden may not be minimized while minimizing the epidemic duration in simple SIR models.

Minimizing the infectious burden in the optimal control problem for isolation and reduction of transmission is equivalent to maximize the final number of susceptibles, $S(T)$. Some examples of the tension between minimizing the epidemic duration and the infectious burden can be observed in Figs. 3 and 5. In particular, Figs. \ref{fig:q}(c) and \ref{fig:r}(c) display the eradication time, $T$, and Figs. \ref{fig:q}(d) and \ref{fig:r}(d) display the number of susceptible individuals at the end of the epidemic, $S(T)$, as functions of $u^{max}$ for both the time-optimal control and the constant control (corresponding to the optimal solution for the unconstrained problem of infectious burden minimization). From these figures, we notice that the different objective functions provide similar results when the control efforts are sufficiently large to rapidly lead the epidemic to extinction (high $u^{max}$), while they provide substantially different results in the case of less efficient strategies (low $u^{max}$). Specifically, the time-optimal control strategy performs poorly in minimizing the infectious burden at the boundary between delayed and constant control (see Figs. \ref{fig:q}(d) and \ref{fig:r}(d)), while the infectious burden minimization strategy performs poorly in minimizing the epidemic duration for slightly higher values of $\R_C$ (see the peak of $\Tumax$ in Figs. \ref{fig:q}(c) and \ref{fig:r}(c)).

Moreover, we find that small changes in the control parameter $u^{max}$ can cause large changes in the shape of the optimal strategies. An analogous result was found by Hansen \& Day \citep{hansen11JOMB} investigating the problem of minimizing the infectious burden through isolation in a SIR framework with limited resources. Hansen \& Day \citep{hansen11JOMB} also found that a ``catastrophic'' shift in the shape of the isolation strategy corresponds to an abrupt variation in the objective function (i.e. the infectious burden). Conversely, here we find that ``catastrophic'' shifts in the shape of the control strategies correspond to continuous variations in the objective functions (i.e. the final time of epidemics).

We believe our findings can be useful in throwing light on overlooked results obtained with more complex models developed in specific epidemiological contexts. For instance, Roche et al. \citep{roche15} investigated the performances of different spatially explicit models for the spread of foot-and-mouth disease in the UK farms, considering different control scenarios. Among other scenarios, they compared the effect of suppressive vaccination strategies started at 7 and 14 days after the outbreak beginning. They found that, in two out of the four models investigated, the medians and/or the $95^{th}$ percentiles of the epidemic duration decreased when the control is delayed by 7 days \citep[][see models `IS+' and `NL' in table 4 therein]{roche15}. On the other hand, they found that the number of infected farms always increases when the vaccination is delayed \citep[see table 4 in][]{roche15}. In a similar way, by investigating the effectiveness of combined culling and movement restriction to control classical swine fever in Switzerland pig farms, D\"{u}rr et al. \citep{durr13} found that delaying the starting of the control from 6 to 16 days after the outbreak beginning reduced the median outbreak duration in three out of the eight analysed scenarios \citep[see figure 4 in][]{durr13}.

Previous works have already shown that delayed control might represent an optimal strategy in some epidemiological applications. For instance, Handel et al. \citep{handel07} and Hansen \& Day \citep{hansen11PRSB} showed that delaying the controls may be optimal in preventing the re-emergence of the epidemic or the emergence of resistant epidemics. Bolzoni et al. \citep{bolzoni14} showed that the delayed control may be optimal in wildlife diseases where the host population growth is density-dependent.

The numerical analyses performed here under the assumption of constant control highlighted that increasing the control efforts may lead to a substantial increase of the eradication time. This is especially true in the case of isolation and reduction of transmission, where the eradication time may increase from two- to five-fold with respect to the ``do-nothing'' alternative (see Figs. \ref{fig:q}(c) and \ref{fig:r}(c)). Similar negative effects of constant efforts on disease control have also been highlighted when the target of the intervention was the reduction of the number of infected individuals \citep{choisy06, bolzoni07_plos, potapov12, bolzoni13}. All these counter-intuitive findings suggest that the implementation of simple time-dependent strategies may crucially improve the control of infectious diseases.

Other aspects of diseases control implementation that were not included in the present work -- such as combined controls \citep{hansen11JOMB}, the costs of control \citep{behncke00}, resources limitation \citep{hansen11JOMB}, and availability of surveillance information \citep{bolzoni07EDE} -- can play a significant role in shaping the optimal strategy. These aspects are essential in defining optimal protocols of intervention for diseases eradication. However, thanks to the generality of the model formulation, we believe our results can be used as a benchmark to contrast the outcome of future investigations. 

\section*{Acknowledgements}
The authors are very grateful to Stefano Pongolini for his invaluable suggestions and remarks that improved the manuscript. LB was supported by the Italian Ministry of Health (grant IZSLER – PRC2014003).
Support by INdAM-GNFM is also gratefully acknowledged by EB, CS and MG.
\\

\appendix
\section{}\label{appendice}

\begin{teoapp}\label{teo:exist} There exists an optimal solution of the optimal control problem \eqref{TOC}. \end{teoapp}
\begin{proof}
By definition, there exists an optimal solution of \eqref{TOC} if the functional $J(u)$, which gives the eradication time of the controlled SIR problem \eqref{dot_x} as a function of the control, has (at least) a minimum point $u^*$ on the set of admissible controls. For each policy Theorem \ref{teo bang-bang} holds, namely the set of admissible controls is $\mathcal{A}$ given in \eqref{A}.
%Moreover, we define the \interv time $\tau$ \eqref{eq:tau} in such a way that each admissible control in $\mathcal{A}$ can be identified by $\tau$ (see \eqref{eq:utau}).
As detailed in Section \ref{stan}, we can see $J$ as a function that links the \interv time $\tau$ \eqref{eq:tau} to the eradication time $T$:
\begin{align*}
J : [0,\tau^{max}] &\rightarrow [0,+\infty)\\
		\tau \quad &\rightarrow \quad T
\end{align*}
Since the \interv time cannot be larger than the eradication time of the uncontrolled epidemic, then there exists an upper bound $\tau^{max}$ for $\tau$. We prove that $J$ admits at least a minimum point $\tau^*$ by proving that it is a continuous function on the bounded interval $[0,\tau^{max}]$.

First we prove that $J$ is a continuous function in 0, namely that $\lim_{h\rightarrow0^+} J(h)=J(0)$. We observe that by definition $J(0)$ is the eradication time $T_0 (=T_{\tau=0})$ of the solution $\x(t)$ of the controlled problem
\begin{align*}
\begin{cases}
\dot{\x}(t)=f(\x(t))+u^{max}g(\x(t)), \quad t\geq0; \\
\x(0)=\x_0, \quad \x(T_0)\in \mathcal{C} = \{ (S,I)\,:\; I=\varepsilon \}
\end{cases}
\end{align*}
while $J(h)$ is the eradication time $T_h$ of the solution $$\y(t)= \begin{cases} \y_1(t) & 0\leq t< h\\ \y_2(t) & t\geq h \end{cases}$$ where $\y_1$ is the solution of the uncontrolled problem
\begin{align*}
\begin{cases}
\dot{\y_1}(t)=f(\y_1(t)), \quad 0\leq t\leq h; \\
\y_1(0)=\x_0
\end{cases}
\end{align*}
while $\y_2$ is the solution of
\begin{align*}
\begin{cases}
\dot{\y_2}(t)=f(\y_2(t))+u^{max}g(\y_2(t)), \quad t\geq h; \\
\y_2(h)=\y_1(h), \quad \y_2(T_h)\in \mathcal{C}
\end{cases}
\end{align*}

By the Continuous Dependence on Initial Conditions Theorem, for a generic $t\geq h$ it holds:
\begin{align*}
|| \y(t)-\x(t) ||_\infty &\leq e^{L(t-h)} || \y(h)-\x(h) ||_\infty \\
												 &\leq e^{L(t-h)} || \y(h)-\x_0|| + ||\x_0-\x(h) ||_\infty \\
												 &\leq e^{L(t-h)} (c_{\y} + c_{\x} ) h
\end{align*}
where the last inequality follows from the Mean Value Theorem. This is true in particular for $t=T_h$: $|| \y(T_h)-\x(T_h) ||_\infty\leq \tilde{c}\, h$. Let us consider only the infected component of the two solutions: $I_{\x}(t)$ and $I_{\y}(t)$. Then $| I_{\y}(T_h) - I_{\x}(T_h) |\leq \tilde{c}\, h$, which leads to $| I_{\x}(T_0) - I_{\x}(T_h) |\leq \tilde{c} h$, since $I_{\y}(T_h) = \varepsilon = I_{\x}(T_0)$. This is equivalent to
\[ \lim_{h\rightarrow0^+} I_{\x}(T_h) = I_{\x}(T_0). \]
$I_{\x}(t)$ being a continuous positive function that is strictly monotone in a neighborhood of $T_0$, it is invertible and therefore we can state that $\lim_{h\rightarrow0^+} T_h = T_0, $
namely $\lim_{h\rightarrow0^+} J(h) = J(0)$.
The proof of the continuity of $J$ in a generic \interv time $\tau$ follows from the continuity in 0, using translation arguments.
\end{proof}

\begin{oss}\label{1quadrante} If we consider non-negative initial data $S(0)$ and $I(0)$, then the solution of the differential system \eqref{dot_x} is non-negative at each time $t>0$. \end{oss}
\noindent Indeed, for all the chosen policies, the $I$ axis is a trajectory for the system; the $S$ axis is also a trajectory (for vaccination and culling policies) or is a set of stationary points (for isolation and reduction of transmission policies).

\begin{oss}\label{prop:regione_invariante} For each $k>0$ the set $\mathcal{Q}_k=\{S\geq0,\,I\geq0,\,S+I\leq k\}$ is a positively invariant (trapping) region. \end{oss}
\noindent Using results of Remark  \ref{1quadrante} it is sufficient to prove that for  $S,I>0$  the vector field evaluated on the boundary line $S+I=k$ points towards the internal part of the region $\mathcal{Q}_k$ \citep{guckenheimer}; it is straightforward for each policy since the scalar product of the vector field $f(\x)$ and the outward pointing normal vector of the boundary $\hat{n}=(1,1)^{\sf T}$ is negative in all cases.

\begin{corollario} Given an initial condition $\x_0=(S(0),I(0))\in\mathbb{R}^2_+$, let $\x(t)=(S(t),I(t))$ be the solution of \eqref{dot_x}. Then $I(t)\rightarrow0$ as $t\rightarrow+\infty$ for all control policies. \end{corollario}
\begin{proof} By Remark \ref{prop:regione_invariante} we know that the set $\mathcal{Q}_{N_0}$, where $N_0=S(0)+I(0)$,  is a   (trapping) region. Moreover, in this region the function $\dot{S}(t)$ has a constant negative sign, so there cannot be periodic trajectories and all orbits must converge to a stationary point $\bar{\x}\in\mathcal{Q}_{N_0}$. It is easy to prove that the number of infected individuals of a stationary point is always zero. In fact, for the vaccination and culling policies, the only stationary point is $\bar{\x}=(0,0)$, while for isolation and reduction of transmission policies the stationary points are all those of the $S$ axis.
\end{proof}

\section{}
Throughout all the proofs we omit the superscript $*$ for the optimal quantities, in order to simplify the notation.

%--------------- DIMOSTRAZIONE CONGIUNTA ---------------
\subsection{Proof for optimal linear term policies}\label{dimgen}

Let $\x(t)=(S(t),I(t))^{\sf T}$ denote the optimal solution for the control problem with linear term policy, suitable to model vaccination, isolation or culling for proper values of parameters $\alpha_1$ and $\alpha_2$; let  $\uu(t)$  be the control term, $\llambda(t)=(\lambda_S(t),\lambda_I(t))^{\sf T}$  the corresponding adjoint variables and $T$ the optimal eradication time. By the Pontryagin's Minimum Principle, the Hamiltonian function, the switching function and its derivative are respectively:
\begin{gather}
\H(\x,\llambda,\uu) = 1 - (\beta SI + \au\uu S)\ls + (\beta SI-\mu I-\ad \uu I)\li \label{V:H}\\
\psi(\x,\llambda)=-\au S\ls-\ad I\li, \qquad \dot{\psi}(\x,\llambda)=\beta SI(\au\li-\ad\ls),
\label{V:psi}
\end{gather}
and the adjoint variables satisfy the following system of ODEs:
\begin{equation}
\begin{cases}
\dot{\ls} &= (\ls-\li)\beta I+\au\uu\ls\\
\dot{\li} &= (\ls-\li)\beta S+\mu\li+\ad\uu\li.
\end{cases}
\label{V:adjoint}
\end{equation}

The sketch of the proof of Theorem \ref{teo:VIC} is as follows.

(i) First we prove that the control is non-singular, namely that the function $\psi$ vanishes only in isolated points. Suppose in fact that $\psi$ vanishes in an open interval $B$. Then also all the derivatives vanish there and in particular $\psi=\dot{\psi}=0$ in $B$, which yields by some algebra $\ls=\li=0$ by \eqref{V:psi}, since $S,I>0$ when $S(0),I(0)>0$ (see Remark \ref{1quadrante}). This is in contradiction with Theorem \ref{PMP}, as the adjoint variables $\ls$ and $\li$ cannot vanish simultaneously by construction, therefore $\psi$ vanishes only in isolated points. As a consequence, the control is a piecewise constant function $\uu(t)$ that can assume only two values: 0 and $\umax$. The switching times are defined as the time instants at which the function $\psi(t)$ changes its sign and, consequently, the function $\uu(t)$ changes its value. Therefore two types of switch can occur: either the value of $\uu(t)$ is 0 in a left-neighbourhood of the switching time and is $\umax$ in a right-neighbourhood, and we denote it by $0\rightarrow \umax$, or the converse, which is denoted by $\umax \rightarrow 0$.

(ii) Next we show that the optimal control in a left-neighbourhood of the eradication time $T$ must be equal to $\umax$. By condition 3 of Theorem \ref{PMP} $\psi(T)=-\ad I(T)\li(T)$ and $\dot{\psi}(T)=\au\beta S(T)I(T)\li(T)$. The sign of the function $\psi$ in the left-neighbourhood of $T$ will then be determined by the sign of $\li(T)$. Substituting $\ls(T)=0$ in \eqref{V:H} and by condition 2 of Theorem \ref{PMP} we get $\li(T)=-\dot{I}(T)^{-1}$, which is positive, since $\dot{I}(T)<0$. As a consequence, $\psi(T)\leq0$ and $\dot{\psi}(T)\geq0$. Since they cannot vanish simultaneously, as $\au$ and $\ad$ are not simultaneously zero, $\psi$ is negative in a left-neighborhood of $T$.

(iii) Now we prove that there can be at most one switching time, relevant to the switch $0\rightarrow \umax$. Let $\tau_s$ be a generic switching time, namely $\psi(\tau_s)=0$. Then $-\au S(\tau_s)\ls(\tau_s)=\ad I(\tau_s)\li(\tau_s)$ by \eqref{V:psi}. Suppose $\ad\neq0$, then at the switching time $\li=-\frac{\au S\ls}{\ad I}$. Substituting this relation in \eqref{V:H} and by condition 2 of Theorem \ref{PMP}, we can write $\ls$, and therefore $\li$ and $\dot{\psi}$, as functions of $Q(t)=\beta I(t) + \frac{\au}{\ad} (\beta S(t) - \mu)$, which is a decreasing function since $\dot{Q}(t)<0$:
%\begin{gather*}
$$\ls(\tau_s)=(Q(\tau_s)S(\tau_s))^{-1}\!\!, \qquad \li(\tau_s)=-\frac{\au}{\ad}(Q(\tau_s)I(\tau_s))^{-1}\!\!, %\\
\qquad
\dot{\psi}=-\frac{\beta}{Q} \left( \frac{\au^2}{\ad}S+\ad I\right).
%\end{gather*}
$$
In particular, we can see that the sign of $\dot{\psi}$ is opposite to the sign of $Q$. Suppose that there are multiple switching times $\tau_s^{(j)}$, $j=1,\ldots,n$. We have already proved that $\uu(T)=\umax$, so at the last switching time $\dot{\psi}(\tau_s^{(n)})<0$ must hold, thus $Q(\tau_s^{(n)})>0$. Since $Q$ is a decreasing function, this means that $Q$ is positive in the interval $[0,\tau_s^{(n)}]$, and it implies that all the switching times $\tau_s^{(j)}$ are from no control (positive values of $\psi$) to maximum control (negative values of $\psi$). This is not possible, therefore there can be at most a unique switch from no control to the maximum control rate $\umax$, namely $\dot{\psi}(\tau_s)<0$.

Now suppose $\ad=0$. We still prove that there can be at most one switching time, relevant to the switch $0\rightarrow \umax$ and that, in addition, the switch can only occur before the peak time. Let $\tau_s$ be a generic switching time, namely $\psi(\tau_s)=0$. Then $\ls(\tau_s)=0$ by \eqref{V:psi} and, analogously to what happens at the eradication time $T$, $\li(\tau_s)=-\dot{I}(\tau_s)^{-1}$. Thus, the sign of $\dot{\psi}(\tau_s)$ is opposite to the sign of $\dot{I}(\tau_s)$, which is positive (respectively, negative) before (resp. after) the peak time of the infection $t_p$.
The only possible change in sign of the function $\psi$ after the peak is then from negative to positive, which is not admissible, since we proved that $\psi$ is negative in a left-neighborhood of $T$. Thus the switch can only occur before the peak and, being $\dot{\psi}(\tau_s)$ always negative, it must be unique. Moreover, since $\psi$ changes its sign from positive to negative values, the control switches from 0 to the maximum rate $\umax$.

\subsection{Proof for optimal vaccination policy}\label{propvacc}

The sketch of the proof of Theorem \ref{teo:vacc} is as follows.

(i) The position of the switch with respect to the peak of infection follows from  \ref{dimgen}, point (iii)  (case $\ad=0$).

(ii)  By definition of the basic reproduction number, we know that if $\R_0<1$ then the number of infected individual is monotonically decreasing in time, namely the peak of the infection is $t_p=0$. As we already proved that there cannot be a switch for $t>t_p$, in this case the optimal control must be $\uv(t)\equiv \vmax$.

Suppose that $\R_0>1$ and that the optimal control is delayed with switch $0\rightarrow \vmax$ at time $\tau_s>0$. Then we prove that the relation $\mu>\vmax$ must hold.

First we prove that, under those hypotheses, the function $\psi$ has a minimum point $m_\psi$ in $(\tau_s,T)$ at which $\li(m_\psi)<\ls(m_\psi)$, as sketched in Fig. \ref{fig:dimostrazione}. The function $\psi$ vanishes at $\tau_s$ (by definition) and at $T$ (by the transversality condition), while it is strictly negative between the two points, therefore it must have at least a minimum point. Since $\psi$ is a $\mathcal{C}^1$ function, in such points $\dot{\psi}=0$ and thus also $\li=0$, by \eqref{V:psi}. Substituting this latter value in the second derivative of the switching function and recalling the definition of $\psi$ in \eqref{V:psi} we obtain $$\ddot{\psi} = \beta S I (\beta S \ls - \beta I \li -u^{max}\li) =\beta^2S^2I\ls = -\beta^2 SI\psi.$$
%, it is easy to see that sign of $\psi$ is opposite to the sign of $\ls$ and therefore $\ls>0$ on $(\tau_s,T)$.
 Then $\ddot{\psi}$ is positive and we can state that $\psi$ has only an extremal point in that interval, and more specifically that it is a minimum point, which we denote by $m_\psi$. Moreover, since $\li(m_\psi)=0$ and $\ls(m_\psi)>0$, it is straightforward that $\li(m_\psi)<\ls(m_\psi)$.

Similarly, we prove that the function $\ls$ has a unique maximum $M_{\ls}$ in the interval $[\tau_s,T]$ at which $\li(M_{\ls})>\ls(M_{\ls})$, as sketched in Fig. \ref{fig:dimostrazione}. In fact, it vanishes at $\tau_s$ (since $\psi(\tau_s)=0$) and at $T$ (for the transversality condition) and it is strictly positive between the two points, since $\psi<0$. On the interval $[\tau_s,T]$, being $u(t)=\vmax$, $\ls$ is a $\mathcal{C}^1$ function, therefore its maximum and minimum points are characterized by $\dls=0$, namely $\ls=\beta I \li/(\beta I+\vmax)$ from \eqref{V:adjoint}. Substituting this value in \eqref{V:H} and recalling that $\H=0$ we obtain that in the extremal points $\li=1/(\mu I)$. Substituting those values in the second derivative of $\ls$ we obtain:
\[ \ddot{\lambda}_S = \ls[(\beta I+\vmax)^2 -\mu\beta I] -\li \beta I(\beta I +\vmax) = -\beta^2 I/(\beta I+\vmax),\]
which is negative and therefore in the interval $(\tau_s,T)$ the function $\ls$ has a unique maximum point, which we denote by $M_{\ls}$. Moreover $$\ls(M_{\ls})=\beta I(M_{\ls}) \li(M_{\ls})/(\beta I(M_{\ls})+\vmax)<\li(M_{\ls}).$$
	
	Evaluating $\dls$ at the point $m_\psi$, by \eqref{V:adjoint} we obtain $$\dls(m_\psi)=(\beta I(m_\psi)+\vmax)\ls(m_\psi)>0,$$ therefore $m_\psi<M_{\ls}$. Being $\li(m_\psi)<\ls(m_\psi)$ and $\li(M_{\ls})>\ls(M_{\ls})$, there must exist a point $\sigma\in(m_\psi,M_{\ls})$ such that $\ls(\sigma)=\li(\sigma)$ and $\dli(\sigma)>\dls(\sigma)$, as sketched in Fig. \ref{fig:dimostrazione}. This last inequality reduces to $\mu\li(\sigma)>\vmax\ls(\sigma)$, therefore $\mu>\vmax$ is a necessary condition for having a positive switching time. In conclusion, if $\vmax>\mu$ the optimal control is the constant control $\uv(t)\equiv \vmax$.
	
	\begin{figure}\centering
	\includegraphics[width=.7\columnwidth]{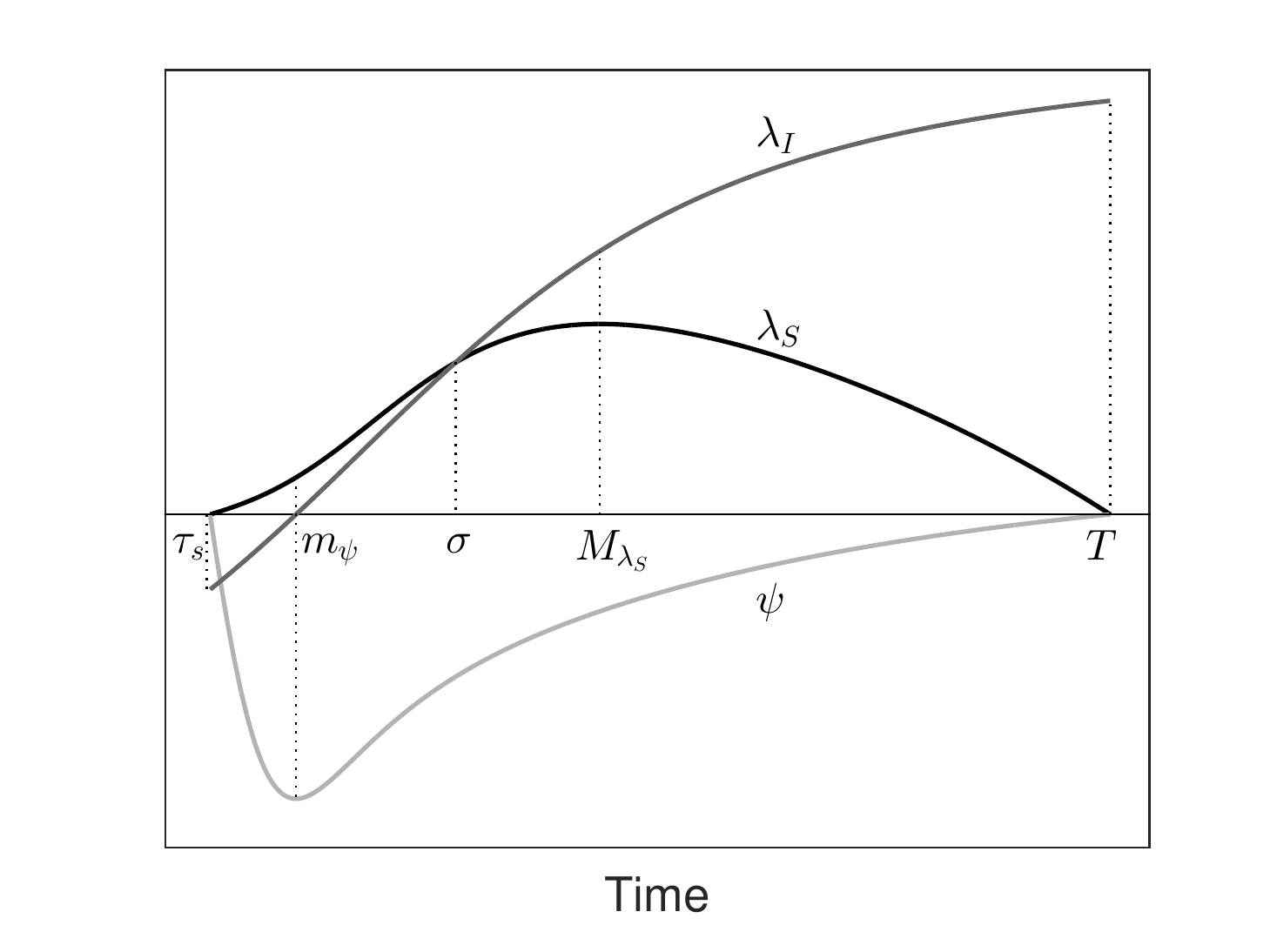}%
	\caption{Optimal vaccination problem with delayed control. Schematization of the switching function $\psi(t)$ and of the adjoint variables $\ls(t)$, $\li(t)$ on the interval $[\tau_s,T]$.}%
	\label{fig:dimostrazione}%
	\end{figure}

\subsection{Proof for optimal reduction of transmission policy}\label{dimR}

Let $\x(t)=(S(t),I(t))^{\sf T}$ denote the optimal solution for the reduction of the transmission control problem, with control term $\ur(t)$,  $\llambda(t)=(\lambda_S(t),\lambda_I(t))^{\sf T}$ the corresponding adjoint variables and $T$ the optimal eradication time. By the Pontryagin's Minimum Principle, the Hamiltonian function, the switching function and its derivative are respectively:
\begin{gather}
\H(\x,\llambda,\ur) = 1 + (\li-\ls)\beta(1-\ur) SI -\mu I\li \label{R:H}\\
\psi(\x,\llambda)=(\ls-\li)\beta SI, \qquad \dot{\psi}(\x,\llambda)=-\mu\beta SI\ls,
\label{R:psi}
\end{gather}
and the adjoint variables satisfy the following system of ODEs:
\begin{equation*}
\begin{cases}
\dot{\ls} &= (\ls-\li)\beta(1-\ur) I\\
\dot{\li} &= (\ls-\li)\beta(1-\ur) S+\mu\li.
\end{cases}
\end{equation*}

For the proof of Theorem \ref{teo:rid}, first we show that the control is non-singular, namely that the function $\psi$ vanishes only in isolated points. Suppose that $\psi$ vanishes in an open interval $B$. Then $\psi=\dot{\psi}=0$ in $B$, namely $\ls=\li=0$ (see \eqref{R:psi}), which is in contradiction with the statement of the Theorem \ref{PMP}. Therefore, $\psi$ can vanish only in isolated points.
Substituting $\ls(T)=0$ (the transversality condition) in \eqref{R:H} and by condition 2 of Theorem \ref{PMP} we get $\li(T)=-\dot{I}(T)^{-1}$, which is positive, being $\dot{I}(T)<0$. As a consequence, $\psi(T)<0$ by \eqref{R:psi}.

Let $\tau_s$ be a generic switching time, namely $\psi(\tau_s)=0$. Then $\ls(\tau_s)=\li(\tau_s)$ by \eqref{R:psi} and, by equation \eqref{R:H}, they are both equal to $(\mu I(\tau_s))^{-1}$. Substituting this value in \eqref{R:psi} we obtain $\dot{\psi}(\tau_s)=-\beta S(\tau_s)$, which is negative. Therefore, since the sign of the derivative of $\psi$ is constant at every switching time $\tau_s$, there can be at most a unique switch from no control (positive values of $\psi$) to the maximum rate of control $\rmax$ (negative values of $\psi$).

\newpage
%\section*{References}
\bibliographystyle{elsarticle-num}
\bibliography{MyBib}

\end{document}